\documentclass[12pt]{amsart}

\usepackage[margin=1.2in]{geometry}

\usepackage{amssymb, xypic, amsmath, amscd,amsthm}
\usepackage{amsxtra}
\usepackage[pdftex]{graphicx}
\xyoption{all}

\newtheorem{thm}{Theorem}[section]
\newtheorem{prop}[thm]{Proposition}
\newtheorem{lem}[thm]{Lemma}
\newtheorem{cor}[thm]{Corollary}

\theoremstyle{definition}
\newtheorem{definition}[thm]{Definition}

\theoremstyle{remark}
\newtheorem{remark}[thm]{Remark}

\numberwithin{equation}{section}

\def\lr#1{\langle#1\rangle}
\def\blr#1{\big\langle#1\big\rangle}
\def\ri{\rightarrow}
\def\Si{\Sigma}
\def\id{\textnormal{id}}
\def\ev{\textnormal{ev}}
\def\wt#1{\widetilde{#1}}
\def\cD{\mathcal{D}}
\def\cJ{\mathcal{J}}
\def\gm{\gamma}
\def\U{\textnormal{U}}
\def\SO{\textnormal{SO}}
\def\eps{\epsilon}
\def\prt{\partial}
\newcommand{\C}{\mathbb{C}}  
\newcommand{\mbp}{\mathfrak{B}_{l,{\bf k}}^{g,h}(M,L, {\bf b})}
\newcommand{\mmp}{\mathfrak{M}_{l,{\bf k}}^{g,h}(M,L, {\bf b})}
\newcommand{\dmp}{\mathfrak{M}_{l,{\bf k}}^{g,h}}
\newcommand{\mhp}{\mathcal{H}_{l,{\bf k}}^{g,h}(M,L, {\bf b})}
\newcommand{\db}[2]{\bar{\partial}_{(#1, #2)}}  
\newcommand{\bp}{\bar{\partial}} 
\newcommand{\lL}{\mathcal{L}(L)} 
\newcommand{\tz}[1]{\mathcal{Z}_{#1}}

\newcommand{\R}{\mathbb{R}}

\DeclareMathOperator{\Ker}{Ker}
\DeclareMathOperator{\coker}{Coker}

\DeclareMathOperator{\aut}{Aut}

\DeclareMathOperator{\Hom}{Hom}

\begin{document}

\title[The orientability problem in open Gromov-Witten theory] {The
orientability problem in \\open Gromov-Witten theory}
\author[Penka Georgieva]{Penka Georgieva*} \thanks{*Partially supported by   NSF grants DMS-0605003 and DMS-0905738.}
\address{Department of Mathematics, Princeton University, Princeton, NJ 08544}
\email{pgeorgie@math.princeton.edu}

\begin{abstract} We give an explicit formula for the holonomy of  the
orientation bundle of a family of real Cauchy-Riemann operators. A special case
of this formula   resolves the orientability question  for spaces of
maps from Riemann surfaces with Lagrangian boundary condition. As a corollary,
we show that the local system of orientations on the moduli space of
$J$-holomorphic maps from a bordered Riemann surface to a symplectic manifold
is isomorphic to the pull-back of a local system defined on the product of the
Lagrangian and its free loop space. As another corollary, we show that certain natural
bundles over these moduli spaces have the same local systems of orientations as the moduli
spaces themselves (this is a prerequisite for integrating the Euler classes of these bundles). 
We will apply these conclusions in future papers to construct and compute open Gromov-Witten invariants
in a number of settings.
\end{abstract}

\maketitle

\tableofcontents
\section{Introduction}
 The theory of $J$-holomorphic maps
introduced by Gromov \cite{Gr} plays a central role in the
study of symplectic manifolds. Considerations in theoretical physics led to the
development of the Gromov-Witten
invariants. They are invariants of symplectic manifolds and can be interpreted
as  counts of $J$-holomorphic maps from a
closed Riemann surface passing
through prescribed constraints. Open String Theory motivated the study of
$J$-holomorphic maps
from a bordered Riemann surface with boundary mapping to
 a Lagrangian submanifold and predicts the existence of open
Gromov-Witten invariants. Their mathematical definition, however, has proved to
be
a subtle problem. Two main obstacles are the question of  orientability
and the existence of real codimension one boundary strata of the moduli space of
maps from a bordered Riemann surface. This work addresses the first of these
issues.\\

The orientability question in the case $\Sigma=D^2$ was previously studied in
\cite{FOOO}; see also \cite{EES}, \cite{WW}, \cite{Wel}.
The authors showed that the moduli space of $J$-holomorphic maps from $D^2$  is
not always orientable; see also \cite{deS}. However,  in the case
of a relatively spin Lagrangian, they proved that the moduli space
  is orientable and that a choice of a relatively spin structure determines an
orientation. This result was extended by Solomon \cite{Sol} to
 relatively pin$^\pm$ Lagrangians and Riemann surfaces of higher genus with a
fixed complex structure. Solomon constructed a canonical isomorphism
between the determinant line bundle of the moduli space and the pull-back by the
evaluation maps of a certain number of
 copies of $\det(TL)$.   We extend these results to any Lagrangian and allow the
complex structure on the domain to vary.\\

In this paper, we give an explicit  criterion specifying whether the determinant
line bundle of a loop of real Cauchy-Riemann operators over   bordered Riemann
surfaces is  trivial; see Theorem \ref{main_thm}.   As a corollary, we conclude
that the local system of orientations on the moduli space of $J$-holomorphic
maps from a bordered Riemann surface
is isomorphic to the pull-back of a local system defined on the product of the
Lagrangian and its free loop space; see Corollary \ref{ls_cor}.  Our formula
recovers the orientability results obtained in \cite{FOOO} and \cite{Sol} as special cases. As another corollary, we show that the local systems of orientations of certain natural
bundles over these moduli spaces are canonically isomorphic to the  local systems of orientations of the moduli
spaces themselves which is a prerequisite for integrating the Euler classes of these bundles; see Corollary \ref{some_cor}. This is a generalization of \cite[Lemma 12]{psw}.\\

If $M$ is a manifold, possibly with boundary, or a (possibly nodal) surface, and $L\subset M$ is a submanifold,
  a \textsf{bundle pair} $(E,F)\ri(M,L)$   consists of a complex vector
bundle $E\ri M$ and a maximal totally real subbundle $F\subset
E_{|L}$. A \textsf{real Cauchy-Riemann
operator} on a bundle pair $(E,F)\ri(\Si,\partial \Si)$,  where $\Si$ is an oriented surface with boundary $\prt\Si$, is a linear map of the
form
$$D=\bp+A:\Omega^0(E,F)\ri\Omega^{0,1}(E),$$
where $\bp$ is the holomorphic $\bp$-operator for some complex structure $j$ on
$\Si$ and a holomorphic structure in $E$ and  $$A\in\Gamma(\Si, \Hom_\R(E,
T^*\Si\otimes_\C E))$$ is a zeroth-order deformation term. All real
Cauchy-Riemann operators are Fredholm in appropriate completions; see
\cite[Theorem C.1.10]{MS}. \\

Let $I=[0,1]$. Given an orientation-preserving diffeomorphism $\phi\!:\Sigma\rightarrow\Sigma$, let
$$(M_\phi,\partial M_\phi)=\big((\Sigma,\partial\Sigma)\times
I\big)/(x,1)\sim(\phi(x),0)$$
be the mapping torus of $\phi$ and $\pi\!: M_\phi\ri S^1$ be the projection map.
For each $t\in S^1$, let $\Si_t=\pi^{-1}(t)$ be the fiber over $t$.   A continuous family of real Cauchy-Riemann
operators on $(E,F)$ is a collection of real Cauchy-Riemann operators
 $$D_t:\Omega^0(E_{|\Si_t},F_{|\partial \Si_t})\ri\Omega^{0,1}(E_{|\Si_t})$$
which varies continuously with $t\in S^1$. We denote by $\det(D)\ri S^1$ the
determinant line bundle corresponding to this family; see \cite[Section
A.2]{MS}.

\begin{thm} \label{main_thm}
Let $\Si$ be a smooth oriented bordered surface, $\phi\!:\Si\ri\Si$ be a diffeomorphism
preserving the orientation and each boundary component, and $(E,F)$ be a bundle pair over $(M_\phi,
\partial M_\phi)$. For each boundary component  $(\partial\Si)_i$ of $\Si$,
choose a section $\alpha_i$ of $$(\partial M_\phi)_i=
M_{\phi_{|(\partial\Si)_i}}\ri S^1.$$ If  $D$ is any family of real
Cauchy-Riemann operators on $(E,F)$, then
\[
 \lr{w_1(\det D), S^1}=\sum_{i} \big(\blr{w_1(F), (\partial \Si)_i}+1\big)
\lr{w_1(F),\alpha_i}+
\sum_{i} \blr{w_2(F),(\partial M_\phi)_i}.
\]
\end{thm}

\begin{remark}
 Any two choices of sections $\alpha_i$ differ by a multiple of $(\prt\Si)_i$
and thus give the same $i$-th term in the first sum above. 
\end{remark}

We prove this theorem in Section \ref{sec:w1}  by showing that the determinant
line bundle of $D$ is isomorphic to the tensor product of  the determinant line
bundle of a $\bp$-operator on a line bundle and the determinant line bundle of a
$\bp$-operator on an orientable bundle. The evaluation of their first
Stiefel-Whitney classes   then gives the two parts in the formula in
Theorem \ref{main_thm}. 

\begin{remark}\label{fam_rem} Families of real Cauchy-Riemann operators often arise by pulling back data from
a target manifold by smooth maps as follows. Suppose $(M, J)$ is an almost
complex manifold, $L\subset M$ is a submanifold, $(E,\mathfrak{I})\ri M$ is  a
complex vector bundle, and $F\subset E_{|L}$ is a maximal totally real
subbundle. Let $\nabla$ be a connection in $(E, \mathfrak{I})$ and $A\in
\Gamma(M, \Hom_\R(E, T^*M^{0,1}\otimes_\C E))$. For any map $u:
(\Si,\partial\Si)\ri (M,L)$, let $\nabla^u$ denote the induced connection in
$u^*E$ and
$$ A_u=A\circ \partial u\in\Gamma(\Si, T\Si^{0,1}\otimes_\C u^*E).$$
If $u:(\Si,\partial \Si)\ri (M,L)$ and $j$ is a complex structure on $\Si$, the
homomorphisms
$$\bp_u^\nabla=\frac{1}{2}(\nabla^u+\mathfrak{I}\circ \nabla^u\circ j), \,\,D_u\equiv
\bp_u^\nabla+A_u: \Omega^{0}(\Si,\partial\Si; u^*E, u_{|\partial\Si}^* F)\ri
\Omega^{0,1}(\Si, u^*E)$$
are real Cauchy-Riemann operators that form families of real Cauchy-Riemann operators
over families of maps.\end{remark}

Throughout this paper,  we denote by  $(M,\omega)$ a symplectic manifold, by
$L\subset M$ a Lagrangian submanifold, and by $J$ a tame almost complex
structure on $M$.  Fix a tuple of homology classes
\begin{equation}\label{btuple_eq}
{\bf b}=(b,b_1,..,b_h)\in H_2(M,L)\oplus H_1(L)^{\oplus h}, 
\end{equation}
an oriented bordered
  surface
$(\Sigma,\partial \Sigma)$ of genus $g$, an ordering of the boundary components
 $$\partial \Sigma =\coprod_{i=1}^h(\partial \Sigma)_i\cong\coprod_{i=1}^hS^1,$$
a non-negative integer $l$, and a tuple ${\bf k}=(k_1,\,..\,, k_h)\in
\mathbb{Z}^h_+$.
Let
$\mbp$  be the space of tuples $(u, {\bf z}, {\bf x}_1,\,..\,, {\bf x}_h)$, where
\begin{itemize}
\item ${\bf z}$ is a tuple of $l$ interior marked points,
\item ${\bf x}_i$ is a tuple of $k_i$   marked points on   $(\prt\Si)_i$,
\item  $u$ is a map from   $(\Sigma,\partial\Si)$ to
$ (M,L)$,
which represents the class $b\in H_2(M,L)$ and for which the
restriction $u_{|(\partial\Sigma)_i}$ represents the class $b_i\in H_1(L)$.
\end{itemize}
Let $\mathcal{J}_\Si$ be the space of complex structures on $\Si$, $\mathcal{D}$
be the diffeomorphism group of $\Si$ preserving the orientation and each
boundary component, and
$$\mhp=(\mbp\times\mathcal{J}_\Si)/\mathcal{D}.$$

 \begin{remark} \label{mfld_rem} Throughout the paper we assume that the action of $\mathcal{D}$ has no fixed points. 
Thus, the quotient space $\mhp$ is a topological manifold. 
This happens for example if there are sufficiently many marked points.
In applications to more general cases,  this issue can be avoided by working with Prym structures on Riemann surfaces; see \cite{Loo}. 
\end{remark}

 The determinant line bundle of a family of real Cauchy-Riemann operators
$D_{(E,F)}$ on $\mbp\times \cJ_\Si$ induced
by a bundle pair $(E,F)$ as in Remark \ref{fam_rem} descends to a line bundle over
$\mhp$, which we denote by $\det D_{(E,F)}$.
As a direct corollary of Theorem~\ref{main_thm}, we obtain the following result
on its orientability.

\begin{cor} \label{orient_cor} Let $\gamma$ be a loop in $\mhp$ and $\widetilde{\gamma}$ a path in
$\mbp\times\mathcal{J}_\Si$ lifting $\gamma$ such that
$\widetilde{\gamma}_1=\phi\cdot\widetilde{\gamma}_0$ for some $\phi
\in\mathcal{D}$ with $\phi_{|\partial\Si}=\id$. For each boundary component
$(\partial\Si)_i$ of $\Si$, denote by $\alpha_i:S^1\ri L$ and $\beta_i:S^1\times
(\partial\Si)_i\ri L$ the paths traced by a fixed point on $(\partial\Si)_i$ and
by the entire boundary component $(\partial\Si)_i$. Then,
\begin{equation}\label{orient_eq}  \lr{w_1(\det D_{(E,F)}),\gamma}=\sum_{i=1}^h
\big(\blr{w_1(F), b_i}+1\big) \lr{w_1(F),\alpha_i}+
\sum_{i=1}^h \blr{w_2(F),\beta_i}.
\end{equation}
\end{cor}

By Lemma \ref{lift_cor}, every loop $\gamma$ in $\mhp$ admits a lift $\widetilde{\gamma}$ such that
$\widetilde{\gamma}_1=\phi\cdot\widetilde{\gamma}_0$ for some $\phi
\in\mathcal{D}$ with $\phi_{|\partial\Si}=\id$. Thus, the first assumption in
Corollary \ref{orient_cor} imposes no restriction on $\gamma$.

\begin{cor}\label{rp} Let $\gamma$ and $\alpha_i$ be as in Corollary \ref{orient_cor}. If either $w_2(F)+w_1^2(F)$ or $w_2(F)$ belong to $ \textnormal{Im}(i^*\!:H^2(M)\ri H^2(L))$, then
 \begin{equation}\label{pin_eq} \lr{w_1(\det(D_{(E,F)})),\gamma}=\sum_{i=1}^h(\lr{w_1(F), b_i}+1)\lr{w_1(F)), \alpha_i}.\end{equation} In
particular, if $F$ is also orientable or $\lr{w_1(F),b_i}=1$ for every $i=1,\dots, h$, then   $\det(D_{(E,F)})$ 
is orientable.\end{cor}

The presence of $w_2(F)$ in
(\ref{orient_eq}) means that in  general the local system of orientations on $\det(D_{(E,F)})$ is not the pull-back of a system on $L$.  
In Section \ref{sec:ls},
we construct
a local system $\mathcal{Z}^F_{(w_1,w_2)}$ on the $h$-fold product of the
Lagrangian and its free loop space $\lL$, which traces the twisting coming from
the right-hand side of (\ref{orient_eq}). When there is at least one boundary
marked point on each boundary component, there is a natural map
$$\ev:\mhp\ri (L\times\lL)^h,$$
which is canonically determined up to homotopy; see Proposition \ref{evmap_prop}.
  We show that the pull-back of $\mathcal{Z}_{(w_1,w_2)}^F$ under this map
is isomorphic to the local system twisted by the first Stiefel-Whitney class of
$\det(D_{(E,F)})$. Moreover, $\ev^*\mathcal{Z}^F_{(w_1,w_2)}$ is trivial along the
fiber of the map forgetting the boundary marked point(s) and pushes down to the
space with no boundary marked points; see Lemma \ref{push_lm}.  Depending on the context, denote by
$\widetilde{\ev}^*\mathcal{Z}^F_{(w_1,w_2)}$ either  the pulled-back system or its
push-down under the forgetful map.

\begin{thm} \label{orient_thm}  There is a local system
$\mathcal{Z}^F_{(w_1,w_2)}$ on $(L\times \mathcal{L}(L))^h$ such that the local
system of orientations
 of $\det D_{(E,F)}$ is isomorphic to $\wt\ev^*\mathcal{Z}^F_{(w_1,w_2)}$. An isomorphism between the two systems is
 determined by a choice of a trivialization of $F$ over a basepoint in $L$ and
trivializations of
$F\oplus 3\det(F)$ over representatives for the homotopy classes of loops in
$L$.
\end{thm}

The $\bp$-operator on $\mbp\times\mathcal{J}_\Si$,  given by
$$\bp(u,j)=\frac{1}{2}(du+J\circ du\circ j),$$
descends to $\mhp$. The moduli space $\mathfrak{M}^{g,h}_{l,{\bf k}}(M,L,\mathbf{b})\subset \mhp$ consists of elements  $[u,{\bf z},{\bf x_1},\dots,{\bf x_h}]$
satisfying $\bp u=0$.
Linearizations of the $\bp$-operator along $\mbp$ are real Cauchy-Riemann
operators over $\Si$ induced by the bundle pair $(TM, TL)\ri (M,L)$; see
\cite[Section 3.1]{MS}. Their determinant line bundle descends to  a line bundle
over $\mhp$, which we denote by
$\det(D_{\bp})$.
The significance of this bundle comes from the fact that when the moduli space
is cut transversely, the top exterior power of its tangent bundle is
essentially the bundle $\det(D_{\bp})$.
As a   corollary of Theorem \ref{orient_thm}, we obtain the following statement concerning the orientation system of this moduli space.

\begin{cor} \label{ls_cor}  There is a local system
$\mathcal{Z}^{TL}_{(w_1,w_2)}$ on $(L\times \mathcal{L}(L))^h$ such that the local
system of orientations
 on the moduli space $\mathfrak{M}^{g,h}_{l,{\bf k}}(M,L,\mathbf{b})$
is isomorphic to $\widetilde{\ev}^*\mathcal{Z}^{TL}_{(w_1,w_2)}$.
An isomorphism between the two systems is
 determined by a choice of a trivialization of $TL$ over a basepoint in $L$ and
trivializations of
$TL\oplus 3\det(TL)$ over representatives for the homotopy classes of loops in
$L$.
\end{cor}

\begin{remark}By Corollaries \ref{rp} and \ref{ls_cor}, $\mathfrak{M}^{g,h}_{l,{\bf k}}(M,L,{\bf b})$ is orientable if $L\subset M$ is relatively spin; this recovers \cite[Theorem 8.1.1]{FOOO}.  If $L\subset M$ is relatively pin$^\pm$, the orientation system of $\mathfrak{M}^{g,h}_{l,{\bf k}}(M,L,{\bf b})$ is a pull-back/push-down of    several copies of the orientation system of the
Lagrangian~$L$; this recovers \cite[Theorem 1.1]{Sol}.
\end{remark}

An important collection of examples of operators $D_{(E,F)}$ arises as follows.
If ${\bf a}=(a_1,\dots, a_m)$ is an $m$-tuple of positive integers and
$n\in\mathbb{Z}^+$, let
$$ \mathcal{L}_{n,{\bf a}}=\mathcal{O}_{\C\mathbb{P}^n}(a_1)\oplus \dots \oplus
\mathcal{O}_{\C\mathbb{P}^n}(a_m)\ri \C\mathbb{P}^n.$$
The natural conjugation on $\C\mathbb{P}^n$ lifts to $\mathcal{L}_{n,{\bf a}}$.
Denote its fixed locus by $\mathcal{L}_{n,{\bf a}}^\R$; this is a real vector
bundle over $\R \mathbb{P}^n$. Let
\begin{equation}\label{Vna_eq}\pi_{n, {\bf a}}\!:\mathcal{V}^\R_{n,{\bf
a}}=\mathfrak{M}^{g, h}_{l,{\bf k}}(\mathcal{L}_{n,{\bf a}},\mathcal{L}_{n,{\bf a}}^\R,{\bf b})\ri
\mathfrak{M}^{g,h}_{l,{\bf k}}(\C\mathbb{P}^n, \R\mathbb{P}^n,{\bf b}). \end{equation}
The fiber of $\pi_{n,{\bf a}}$ over     $[u, {\bf z},{\bf x_1},\dots, {\bf
x_h}]$ is canonically isomorphic to $\Ker \bp_{(u^*\mathcal{L}_{n,{\bf a}},
u^*\mathcal{L}_{n,{\bf a}}^\R)}$. By \cite[Theorem C.1.10(iii)]{MS},
$\bp_{(u^*\mathcal{L}_{n,{\bf a}}, u^*\mathcal{L}_{n,{\bf a}}^\R)}$ is
surjective if $\mu(b)\geq 4g+2h-2$. Thus, $\mathcal{V}^\R_{n,{\bf a}}$ is a vector
bundle in this case and its orientation line bundle agrees with $\det
(\bp_{(u^*\mathcal{L}_{n,{\bf a}}, u^*\mathcal{L}_{n,{\bf a}}^\R)})$.
The following corollary of Theorem~\ref{orient_thm},  suggests that it may be possible to integrate the twisted Euler
class $e(\mathcal{V}^\R_{n,{\bf a}})$ against the homology class of a
compactification  $\wt{\mathfrak{M}}_{l,{\bf k}}^{g,h}(\C\mathbb{P}^n,\R\mathbb{P}^n,{\bf b})$ of $\mathfrak{M}_{l,{\bf k}}^{g,h}(\C\mathbb{P}^n,\R\mathbb{P}^n,{\bf b})$ in
some cases, including when the corresponding complete intersection $X_{n,{\bf a}}$ is a Calabi-Yau threefold; see Remark~\ref{ci_rem}.

\begin{cor}\label{some_cor}
Let $m,n\in\mathbb{Z}^+$ and ${\bf a}\in(\mathbb{Z}^+)^m$ be such that $n-\sum a_i$ is odd. Let ${\bf b}$ be as in (\ref{btuple_eq}) with $(M,L)=(\C\mathbb{P}^n,\R\mathbb{P}^n)$. If $\mu(b)\geq 4g+2h-2$, the line bundles
$$
\Lambda^{\textnormal{top}}_\R \mathcal{V}^\R_{n;{\bf a}},\,\, \Lambda^{\textnormal{top}}_\R T\mathfrak{M}_{l,{\bf k}}^{g,h}(\C\mathbb{P}^n,\R\mathbb{P}^n,{\bf b})\ri \mathfrak{M}_{l,{\bf k}}^{g,h}(\C\mathbb{P}^n,\R\mathbb{P}^n,{\bf b})
$$
are   canonically isomorphic up to multiplication by $\R^+$ in each fiber.  \end{cor}

\begin{remark}\label{ci_rem}
If $s\in H^0(\C\mathbb{P}^n, \mathcal{L}_{n,{\bf a}})$ is a transverse section
commuting with the conjugations on $\C\mathbb{P}^n$ and $\mathcal{L}_{n,{\bf a}}$,
$X_{n,{\bf a}}=s^{-1}(0)$ is a smooth complete intersection with conjugation
inherited from $\C\mathbb{P}^n$. The section $s$ induces a section $\wt s$ of
(\ref{Vna_eq}) such that
$$
\mathfrak{M}^{g,h}_{l,{\bf k}}(X_{n,{\bf a}},X_{n,{\bf a}}^\R,{\bf b}) = \wt{s}^{-1}(0)\subset
\mathfrak{M}^{g,h}_{l,{\bf k}}(\C\mathbb{P}^n,\R\mathbb{P}^n,{\bf b}),
$$
where $X_{n,{\bf a}}^\R=X_{n,{\bf a}}\cap\R\mathbb{P}^n$. 
This suggests that (open) Gromov-Witten invariants of $(X_{n,{\bf a}},X_{n,{\bf a}}^\R)$, which
should arise from the moduli space $\mathfrak{M}^{g,h}_{l,{\bf k}}(X_{n,{\bf
a}},X_{n,{\bf a}}^\R,{\bf b})$, can be computed by integrating the Euler class
$e(\mathcal{V}^\R_{n,{\bf a}})$ against $[\wt{\mathfrak{M}}^{g,h}_{l,{\bf
k}}(\C\mathbb{P}^n,\R\mathbb{P}^n,{\bf b})]$, which can be done via equivariant localization. By
\cite[Section 2.1.2]{Go}, \cite[Theorem 1.1]{LZ}, and \cite[Theorem 3]{psw}, this is indeed the case if
$g+h\leq 1$ and $X_{n,{\bf a}}$ is a Calabi-Yau threefold in the $h=1$ case.
Based on
\cite{Wal} and \cite{PZ}, there are strong indications that this is also the
case for $(g,h)=(0,2)$. 
We plan to investigate this in the future, building on Corollary \ref{some_cor}.
\end{remark}

The paper is organized as follows. In Section \ref{prelim_sec}, we set up the notation and establish some 
preliminary results.  We prove the key Theorem \ref{main_thm}, as well as Corollary \ref{rp},  in Section \ref{sec:w1}.  In Section \ref{sec:ls}, we construct
a local system $\mathcal{Z}_{(w_1,w_2)}^F$ on the $h$-fold product of the
Lagrangian and its free loop space, which traces the twisting coming from the
right-hand side in (\ref{orient_eq}).  We then show that its pull-back
is canonically isomorphic to the local system  twisted by the first
Stiefel-Whitney class of $\det(D_{(E,F)})$, thus establishing Theorem \ref{orient_thm}.   Corollaries  \ref{ls_cor} and \ref{some_cor} are proved at the end of the section.\\

The present paper is based on a portion of the author's thesis work completed at Stanford University. The author would like to thank her advisor Eleny Ionel for  her guidance    and encouragement throughout the years. The author would also like to thank  Aleksey Zinger  for suggesting Corollary \ref{some_cor}  and for his help   with the exposition.

\section{Conventions
 and preliminaries}\label{prelim_sec}

Let $X,Y$ be Banach spaces and $D:X\rightarrow Y$ be a Fredholm operator. The
\textsf{determinant line} of $D$ is defined as \[\det (D):=\Lambda ^{\textnormal{top}}\Ker
(D)\otimes \Lambda^{\textnormal{top}}\coker (D)\spcheck.\]
A short exact sequence of Fredholm operators
\[\begin{CD}
0
@>>>X'@>>>X@>>>X''@>>>0 \\
@. @V V D' V@VV D V@VV D'' V@.\\
0@>>> Y'@>>>Y@>>>Y''@>>>0
\end{CD}\]
 determines a canonical isomorphism
\begin{equation}\label{sum}
\det (D)\cong \det (D')\otimes \det (D'').
\end{equation}
For a continuous family of Fredholm operators $D_t:X_t\ri Y_t$  parametrized by
a topological space $B$, the determinant lines $\det(D_t)$ form a line bundle
over $B$; see \cite[Section A.2]{MS}. For a short exact sequence of such
families, the isomorphisms (\ref{sum}) give rise to a canonical isomorphism
between determinant line bundles. \\

Let $\Sigma$ be a nodal bordered Riemann surface  and $\pi\!:\wt{\Si}\ri\Si$ be
its normalization; fix an ordering of the nodes of $\Si$ and the boundary
components of $\wt{\Si}$.   A real Cauchy-Riemann operator $D_{(E,F)}$ on
$(E,F)\ri\Sigma$ corresponds to a real Cauchy-Riemann operator
$\widetilde{D}_{(E,F)}=\oplus_i D^i$ on $(\wt{E},\wt{F})\equiv
\pi^*(E,F)\ri\widetilde{\Sigma}$, where the sum is taken over the components of
$\wt{\Si}$.
Thus, by \eqref{sum}, there is a canonical isomorphism
\[\det(\widetilde{D}_{(E,F)})\cong\otimes_i\det(D^i)  .\]
On the other hand, by gluing together punctured disks around the special points
in $\widetilde{\Sigma}$, we obtain a smooth surface $\Sigma^\varepsilon$ and a
real Cauchy-Riemann operator $D^\varepsilon$ over $\Sigma^\varepsilon$ for a
gluing parameter $\varepsilon$. By \cite[Section 3.2]{EES} and \cite[Section
4.1]{WW}, for every sufficiently small $\varepsilon$ there is a canonical
isomorphism
\begin{equation}\label{dets}\det(D^\varepsilon)\cong
\det(\widetilde{D}_{(E,F)})\otimes\Lambda^{\textnormal{top}}(\bigoplus_j
E_{z_j}\oplus\bigoplus_j F_{x_j})\spcheck,\end{equation}
where $z_j$ and $x_j$ are the interior and boundary nodes, respectively.
Moreover, the gluing maps satisfy an associativity property: the isomorphism
\eqref{dets} is independent of the order in which we smooth  the nodes.

\begin{remark}\label{rem} Let $(E,F)\rightarrow (\Sigma,\partial\Sigma)$ be a
bundle pair. Choose a trivialization of $E$ over a curve
$C_\epsilon\subset\Sigma$ isotopic to one of the boundary components of
$\Sigma$. This trivialization can  be extended over  a   neighborhood $U$ of the
curve $C_\epsilon$. Pinching $\Sigma$ along $C_\epsilon$, we obtain a nodal curve $\Sigma^s$
with a diffeomorphism $(\Sigma-C_\epsilon)\rightarrow (\Sigma^s - \text{node})$.
We can pull back the bundle pair $(E,F)$ to $(\Sigma^s-\text{node})$. The
trivialization of $E$ over  the neighborhood $U$ of the curve $C_\epsilon$ induces a trivialization
in a punctured neighborhood of the node. It can be uniquely extended over
$\Sigma^s$. We say that the bundle pair $(E,F)$ \textsf{descends} to a bundle
pair on the nodal surface.
\end{remark}

\begin{lem} \label{iso_lm} Let $(\Si,\partial\Si)$ be a smooth oriented surface
with boundary. Every diffeomorphism $h\!:(\Sigma, \partial\Sigma)\rightarrow
(\Sigma,\partial\Sigma)$ which preserves the orientation and  each boundary
component is isotopic to a diffeomorphism which restricts to the identity
on a neighborhood of $\prt\Si$ in $\Sigma$.
\end{lem}

\begin{proof} Fix a component $(\partial\Si)_i\cong S^1$ of $\partial\Si$, an
identification of a neighborhood of $(\partial\Si)_i$ in $\Si$ with $S^1\times
[0,2\delta]$, and $\epsilon\in (0,\delta)$ such that
$h(S^1\times[0,\epsilon])\subset S^1\times [0,2\delta]$. After composing  $h$
with a path of diffeomorphisms on $\Si$ which restrict to the identity outside
$S^1\times [0,2\delta]$, we can assume that $h(S^1\times [0,\epsilon])=S^1\times
[0,\epsilon]$. By \cite[Proposition 2.4]{FM} and \cite[(1.1)]{Mas}, the group of diffeomorphisms of the
cylinder preserving the orientation and each boundary component is path-connected. Thus, there is a path of diffeomorphisms
$$f_t: S^1\times [0,\epsilon]\rightarrow S^1\times
[0,\epsilon]\quad\hbox{s.t.}\quad f_0=\id, ~ f_1=h^{-1}_{|S^1\times
[0,\epsilon]}.$$
The path $f_t$ generates a time-dependent vector field $X_t$. By multiplying
$X_t$  by a bump function on $\Si$ vanishing outside $[0,\epsilon]$ and
restricting to $1$ on $S^1\times [0,\frac{\epsilon}{2}]$, we obtain a
time-dependent vector field $\wt{X}_t$ on $\Si$. This new vector field gives
rise to  diffeomorphisms $\tilde{f}_t$ of $\Si$ which are identity outside
$S^1\times [0,\epsilon]$, while $\tilde{f}_1$ restricts  to $h^{-1}$ on
$S^1\times [0,\frac{\epsilon}{2}]$. Then $h\circ \tilde{f}_t$ is a path of
diffeomorphisms connecting $h$ with a diffeomorphism which restricts to the
identity in a neighborhood of $(\partial\Si)_i$.\end{proof}

\begin{lem}\label{def_lm}
Let $(\Si,\partial\Si)$ be a smooth oriented surface with boundary and $\phi\in
\cD$. Every family of real Cauchy-Riemann operators on a bundle pair $(E,F)$
over $M_\phi$ can be smoothly deformed to a  family of real Cauchy-Riemann
operators on a bundle pair $(E',F')$ over $M_{\phi'}$ for some $\phi'\in\cD$
such that $\phi'$ restricts to the identity on a neighborhood of $\prt\Si$.
\end{lem}

\begin{proof}
By Lemma \ref{iso_lm}, there exists a path $h_s$ in $\cD$ such that $h_0=\phi$
and $h_1$ restricts to the identity on a neighborhood of $\prt\Si$ in $\Si$. Set
$f_s=\phi^{-1}\circ h_s$. Let $(j_t, E_t, F_t, D_t)$, with $t\in I$, be any
family of tuples such that $j_t$ is a complex structure on $\Si$, $D_t$ is a
real Cauchy-Riemann operator on $(E_t,F_t)$ over $(\Si, \prt\Si)$, and
$$
(j_1,E_1,F_1, D_1)=(\phi^*j_0,\phi^*E_0, \phi^*F_0,\phi^*D_0).
$$
 For each $s\in I$, let
 $$
 (j_{s;t}, E_{s;t}, F_{s;t}, D_{s;t})=(f_{st}^*j_t, f_{st}^*E_t, f_{st}^*F_t,
f_{st}^*D_t).
 $$
 Since $(j_{s;1}, E_{s;1}, F_{s;1}, D_{s;1})=(h_s^*j_{s;0},
h_s^*E_{s;0},h_s^*F_{s;0},h_s^*D_{s;0})$, this defines  families of real
Cauchy-Riemann operators on the bundle pairs $(E_s, F_s)$ over $M_{h_s}$. Since
$h_0=\phi$, we have thus constructed the desired deformation of the original
family.
\end{proof}

\begin{lem}\label{lift_cor} Every loop $\gamma$ in $\mhp$ lifts to a path  $\widetilde{\gamma}$ in
$\mbp\times\mathcal{J}_\Si$ such that $\widetilde{\gamma}_1=\phi\cdot
\widetilde{\gamma}_0$ for some $\phi\in\mathcal{D}$ with
$\phi_{|\partial\Si}=\id$.
\end{lem}
\begin{proof}
Under the assumption of Remark \ref{mfld_rem}, the projection $$\mbp\times \cJ_\Si\ri \mhp$$ admits local slices. Thus, there exists a path $\wt\gm_t=(u_t,j_t)$ in $\mbp\times\cJ_\Si$
lifting $\gm$.
Let $\phi\in\cD$ be such that $\wt\gm_1=\phi\cdot\wt\gm_0$.
By Lemma~\ref{iso_lm}, there exists a path $h_t$ in $\cD$ such that $h_0=\phi$
and $h_1$ restricts
to the identity on the boundary.
The lift $\wt\gm'_t=h_t\cdot \phi^{-1}\cdot\wt\gm_t$ of $\gm$ then satisfies
$\wt\gm'_1=h_1\cdot\wt\gm'_0$.
\end{proof}

\section{Determinant line bundles over loops} \label{sec:w1}
We begin this section by deducing Theorem \ref{main_thm} from  Propositions
\ref{mainor_prop} and \ref{mainlb_prop} below, which treat two distinct cases of
Theorem \ref{main_thm}. We then verify each of the two propositions for the
trivial mapping cylinder over the disk with an additional assumption on the
Maslov index of the pair $(E,F)$ on each fiber; see Lemmas \ref{mainord_lm} and
\ref{mainlbd_lm}. The full statements of
Propositions \ref{mainor_prop} and \ref{mainlb_prop} are then reduced to  Lemmas
\ref{mainord_lm} and \ref{mainlbd_lm}, respectively. We conclude this section by proving Corollary \ref{rp}.

\begin{prop} \label{mainor_prop}
Let $\Si$ be a smooth oriented bordered surface, $\phi\!:\Si\ri\Si$ be a
diffeomorphism preserving the orientation and each boundary component, and
$(E,F)$ be a bundle pair over $(M_\phi, \partial M_\phi)$ with $F$  orientable.
If  $D$ is any family of real Cauchy-Riemann operators on $(E,F)$, then
\[
 \lr{w_1(\det D), S^1}=
\sum_{i} \blr{w_2(F),(\partial M_\phi)_i}.
\]
\end{prop}

\begin{prop} \label{mainlb_prop}
Let $\Si$ be a smooth oriented bordered surface, $\phi\!:\Si\ri\Si$ be a
diffeomorphism preserving the orientation and each boundary component, and
$(E,F)$ be a bundle pair over $(M_\phi, \partial M_\phi)$ with $\dim(F)=1$. For
each boundary component  $(\partial\Si)_i$ of $\Si$, choose a section $\alpha_i$
of $$(\partial M_\phi)_i= M_{\phi_{|(\partial\Si)_i}}\ri S^1.$$ If  $D$ is any
family of real Cauchy-Riemann operators on $(E,F)$, then
\[
 \lr{w_1(\det D), S^1}=\sum_{i} \big(\blr{w_1(F), (\partial \Si)_i}+1\big)
\lr{w_1(F),\alpha_i}.
\]
\end{prop}

\begin{remark}\label{con_rmk}
The space of real Cauchy-Riemann operators  $(E,F)\ri(\Si,\partial \Si)$ is
contractible; thus,  a choice of orientation on one determinant line canonically
induces orientations on the rest.  Moreover, any two families of real
Cauchy-Riemann operators on a family $(E_t,F_t)\ri(\Si_t,\partial\Si_t)$ are
fiberwise homotopic. This implies that their determinant bundles have the same
Stiefel-Whitney class.
\end{remark}

\begin{proof} [{\bf \emph{Proof of Theorem \ref{main_thm}}}]
By Remark \ref{con_rmk}, it is sufficient to prove the result for some family of
real Cauchy-Riemann operators on $(E,F)\ri (M_\phi,\prt M_\phi)$.
A connection on $E$ induces a family of complex linear Cauchy-Riemann operators
$\db{E}{F}$ over $M_\phi$, which on a fiber $\Si_t$ is given by the complex
linear Cauchy-Riemann operator of the restricted connection.   Let
\[ (E^1, F^1)=({\textstyle\det_{\C}E},\det  F), \quad (\wt E,\wt F)=(E\oplus
3E^1, F\oplus 3 F^1 ).\]
The connection on $E$ induces connections on $E^1$, $\wt E$, and $4E^1$ and thus
families of complex linear Cauchy-Riemann operators $\db{\wt E}{\wt F}$,
$\db{E^1}{F^1}$, and $\db{4E^1}{4F^1}$.  By \eqref{sum},
\begin{equation*} \begin{split}
\det(\db{\wt E}{\wt F})\otimes \det(\db{E^1}{F^1})&\cong \det(\db{\wt E\oplus
E^1}{\wt F\oplus F^1})
=\det(\db{E\oplus 4E^1}{F\oplus 4F^1})\\&\cong \det(\db{E}{F})\otimes
\det(\db{4E^1}{4F^1}). \end{split}
\end{equation*}\vspace{0.0001 cm}
Therefore, \begin{gather*}\begin{align*}
w_1(\det(D))&=w_1(\det(\db{E}{F}))\\&=w_1(\det(\db{\wt E}{\wt
F}))+w_1(\det(\db{E^1}{F^1}))+w_1(\det(\db{4E^1}{4F^1})).\end{align*}
\end{gather*}
By Proposition \ref{mainor_prop},
\begin{equation*}
 \begin{split} \lr{w_1(\det(\db{\wt E}{\wt F})), S^1}&=\sum_i \lr{w_2(\wt
F),(\prt M_\phi)_i}=\sum_i \lr{w_2 (F),(\prt M_\phi)_i}, \\
                       \lr{w_1(\det(\db{4E^1}{4F^1})), S^1}&=\sum_i
\lr{w_2(4F^1),(\prt M_\phi)_i} =0.
 \end{split}
\end{equation*}
  By Proposition \ref{mainlb_prop},
  \begin{equation*}
  \begin{split} \lr{w_1(\det(\db{E^1}{F^1})),S^1}&=\sum_i(\lr{w_1(F^1),
(\prt\Si)_i}+1)\lr{w_1(F^1),\alpha_i}\\
                              & =\sum_i(\lr{w_1(F),
(\prt\Si)_i}+1)\lr{w_1(F),\alpha_i}.
   \end{split}
 \end{equation*}
   Combining the last four identities, we obtain Theorem \ref{main_thm}.
\end{proof}
 
\begin{lem}\label{mainord_lm}
Let $(E,F)\ri (D^2,\partial D^2)\times S^1$ be a bundle pair with $F$ orientable
and Maslov index zero on each fiber. If  $D$ is a family of real Cauchy-Riemann
operators on $(E,F)$, then
\[
 \lr{w_1(\det D), S^1}=
 \blr{w_2(F),\partial D^2\times S^1}.
\]
\end{lem}
\begin{proof}
The standard $\bp_0$-operator on $(\C^n,\R^n)\ri (D^2,\partial D^2)$ is
surjective and its kernel consists of constant real-valued sections; see
\cite[Theorem C.1.10]{MS}.  If the bundle pair $(E,F)\ri (D^2,\partial
D^2)\times S^1$ is trivializable, we can consider the constant family of
standard $\bp_0$-operators on a trivialization
$$(E,F)\cong (\C^n\times D^2,\R^n\times \partial D^2)\times S^1.$$
The determinant bundle of this family is isomorphic to $\R^n \times S^1 $  by
evaluation at a boundary point
and in particular is orientable. By Remark \ref{con_rmk}, the determinant bundle
of the family $D$ is also orientable. \\

If $(E',F')\ri(D^2,\partial D^2)\times S^1$ is another bundle pair,
$$\det D_{(E,F)}\otimes \det D_{(E',F')}\cong \det (D_{(E,F)}\oplus
D_{(E',F')})\cong\det{D_{(E\oplus E',F\oplus F')}}$$
by (\ref{sum}).
Thus, we can stabilize $(E,F)$ with a trivial bundle pair and assume that
$n=\dim F > 2$.
Since $\pi_1(SO(n))\cong \mathbb{Z}_2$ and the homomorphism
$\pi_1(SO(n))\ri\pi_1 (U(n))$ induced by the inclusion is trivial, the second
Stiefel-Whitney class $w_2(F)$ then classifies  the bundle pairs $(E,F)$ over
$(D^2,\partial D^2)\times S^1$.
Thus, if $w_2(F)=0$, the bundle pair $(E,F)$ is trivializable and the
determinant bundle $\det(D)$ is orientable. If $w_2(F)\neq 0$, the bundle pair
$(E,F)$ is isomorphic to a stabilization of the bundle pair in \cite[Proposition
8.1.7]{FOOO}, which constructs a non-orientable family of real Cauchy-Riemann
operators. Combining the two cases gives the result.
\end{proof}

\begin{lem}\label{lemma}
If $(E,F)\rightarrow (D^2,\partial D^2)$ is a bundle pair with
$\dim_{\mathbb{C}}E=1$ and Maslov index  $\mu=\mu(E,F)\ge -1$, every real
Cauchy-Riemann operator $D$ on $(E,F)$ is surjective. Moreover,  if $x_1,\dots,
x_{\mu+1}\in \partial D^2$ are distinct points, then the homomorphism
\begin{equation} \label{eval_eq}\ev\!: \Ker(D)\ri \bigoplus_{i=1}^{\mu+1}
F_{x_i}, \qquad \ev(\xi)=(\xi(x_1),\dots, \xi(x_{\mu+1})), \end{equation}
is an isomorphism.
\end{lem}

\begin{proof}
By \cite[Theorem C.3.5 and Corollary C.3.9]{MS}, the bundle pair $(E,F)$ is
isomorphic to $(\C\times D^2, \Lambda)$, where the fiber at
$e^{\mathfrak{i}\theta}\in \partial D^2\cong S^1$ is given by
\[\Lambda_{e^{\mathfrak{i}\theta}}=e^\frac{\mathfrak{i}\theta\mu}{2}\mathbb{R}.
\]
By \cite[Theorem C.1.10]{MS}, the standard Cauchy-Riemann operator
$\bar\partial_0$ on $(\C\times D^2, \Lambda)$  is surjective if $\mu\geq -1$
and thus  $\dim\Ker{\bp_0}=\mu+1$. Moreover, the elements of the kernel are
polynomials $\xi(z)=a_0+\dots + a_\mu z^\mu$ with $a_i=\bar{a}_{\mu-i}$.
The kernel of the homomorphism
\begin{equation}\label{eval} \ev\!: \Ker(\bp_0)\ri \bigoplus_{i=1}^{\mu+1}
\Lambda_{x_i}, \qquad \ev(\xi)=(\xi(x_1),\dots, \xi(x_{\mu+1})),
\end{equation}
consists of polynomials (of degree $\mu$) which vanish at the $\mu+1$ points
$x_i$; therefore, this homomorphism is injective. Since the domain and target
are of the same dimension, the homomorphism (\ref{eval}) is an isomorphism.\\

Let $D'$ be any real Cauchy-Riemann operator on the above bundle pair $(\C\times
D^2,\Lambda)$. By \cite[Theorem C.1.10]{MS}, $D'$  is still surjective
and $\dim \Ker(D')=\mu+1$.  If the homomorphism (\ref{eval}) with $\bp_0$
replaced by $D'$ is not an isomorphism, there exists $\xi\in \Ker(D')-\{0\}$
vanishing at the $\mu+1$ points $x_i$. By \cite[Section C.4]{MS}, there exists
$f\!:(D^2,\partial D^2)\ri(\C^*,\R^*)$ such that $\bp_0(f\xi)=0$. Since $f\xi$
vanishes at the $\mu+1$ points, by the previous paragraph $\xi$ is identically
zero. Thus, the homomorphism (\ref{eval}) with $\bp_0$ replaced by $D'$ is in
fact an isomorphism.\\

An isomorphism $\varphi\!: (E,F)\rightarrow (\C\times D^2, \Lambda)$ induces  a
commutative diagram
\[\begin{CD}
\Ker(D)
@>\ev>> \bigoplus\limits_{i=1}^{\mu+1}F_{x_i}  \\
@V\cong V\varphi V @V\cong V\varphi V\\
\Ker((\varphi^{-1})^*D)@>\ev>\cong> \bigoplus\limits_{i=1}^{\mu+1} \Lambda_{x_i}
\end{CD}\]
where   $(\varphi^{-1})^*D$ is the induced real Cauchy-Riemann operator on
$(\C\times D^2, \Lambda)$. Since three of the maps in the diagram are
isomorphisms, so is the evaluation map (\ref{eval_eq}).
\end{proof}

\begin{lem}\label{mainlbd_lm}
Let $(E,F)\ri (D^2,\partial D^2)\times S^1$ be a bundle pair with $\dim(F)=1$
and a non-negative Maslov index $\mu$. If  $D$ is a family of real Cauchy-Riemann
operators on $(E,F)$ and $x\in\partial D^2$, then
 \[\lr{w_1(\det D), S^1}=\big(\blr{w_1(F), \partial D^2}+1\big)
\lr{w_1(F),x\times S^1}.
\]
\end{lem}
\begin{proof}  By Lemma \ref{lemma}, the  operators $D_t$,  $t
\in S^1,$  are surjective and
\[
\ev\!: \Ker(D_t)\ri \bigoplus_{i=1}^{\mu+1} F_{|x_i\times t}, \qquad
\ev(\xi)=(\xi(x_1),\dots, \xi(x_{\mu+1})),
\]
are isomorphisms for any choice of distinct points $x_1,\dots,
x_{\mu+1}\in\partial D^2$.  Therefore,
 \begin{equation*}\begin{split}
\lr{w_1(\det(D)),S^1}=\lr{w_1\left(\bigoplus_{i=1}^{\mu+1} F_{|x_i\times
t}\right), S^1}&= \sum_{i=1}^{\mu+1} \lr{w_1(F),  x_i\times
S^1)}\\&=(\mu+1)\lr{w_1(F),x_1\times S^1}.
\end{split}
\end{equation*}
Since the Maslov index $\mu$ modulo two is $\lr{w_1(F), \partial D^2}$,
$$(\mu+1)\lr{w_1(F),x_1\times S^1}=\big(\lr{w_1(F), \partial
D^2}+1\big)\lr{w_1(F),x_1\times S^1}, $$
establishing the formula.
\end{proof}

\begin{proof}[{\bf \emph{Proof of Proposition \ref{mainor_prop}}}]
By Lemma \ref{def_lm}, we can assume that $\phi$ restricts to the identity in a
neighborhood of the boundary.
For each boundary component $(\partial\Si)_i$ of $\Si$, let
$$U_i\cong[0,2\eps]\times (\prt\Si)_i\times S^1\equiv \textnormal{Cyl}\times S^1
$$
be a neighborhood of $(\prt\Si)_i\times S^1$ in $M_\phi$. Let $j_0$ be a
standard complex structure on $[0,2\eps]\times (\prt\Si)_i$. 
Since every loop
$\gamma$ of complex structures on the cylinder Cyl is of the form
$j_t=\psi_t^*j_0$ for some loop of diffeomorphisms $\psi_t$, there is an
isomorphism
$$(\textnormal{Cyl}\times S^1; \gamma)\cong (\textnormal{Cyl}\times S^1;
j_0),\quad (x,t;j_t)\mapsto (\psi_t(x),t;(\psi^{-1}_t)^* j_t= j_0).$$
Thus, there is an isomorphism
\begin{equation}\label{ui_eq}(U_i, j_{t|U_i})\cong ([0,2\epsilon]\times
(\prt\Si)_i\times S^1, j_0).\end{equation}

For each $\delta\in [0,2\eps]$, let
$$\wt U_i(\delta)\cong[0,\delta]\times(\partial\Sigma)_i\times I$$
 be the neighborhood of $(\partial\Sigma)_i\times I\subset \Sigma\times I$
corresponding to $[0,\delta]\times(\partial\Sigma)_i\times S^1$ under the
identification (\ref{ui_eq}); for example,  $\wt
U_i(0)\cong\{0\}\times(\partial\Sigma)_i\times I$.
We can trivialize $F$ over $\wt U_i(0)$, since $F$ is orientable and $\wt
U_i(0)$ is homotopic  to a circle.
A trivialization $F_{|\wt U_i(0)}\cong \mathbb{R}^{n}\times \wt U_i(0)$ induces
a trivialization
$$E_{|\wt U_i(0)}\cong F\oplus J F_{|\wt U_i(0)}\cong \mathbb{C}^{n}\times \wt
U_i(0),$$ which we can extend to a trivialization $E_{|\wt U_i(2\epsilon)}\cong
\mathbb{C}^{n}\times \wt U_i(2\epsilon)$.

\begin{figure}[htb]
\begin{center}
\leavevmode
\includegraphics[width=1.0\textwidth]{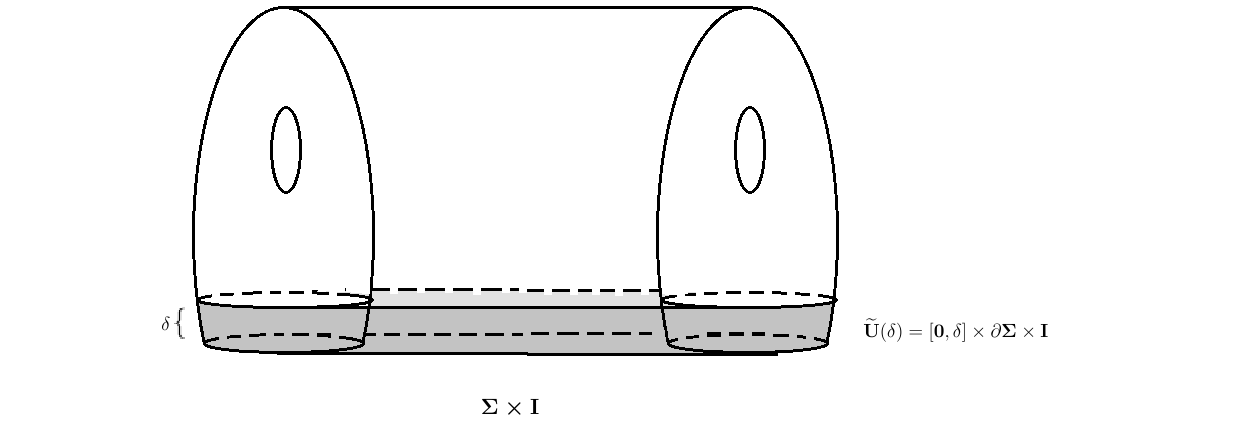}
\end{center}
 \label{fig:def_image}
\end{figure}

At the two endpoints of the interval $I$, glue the trivial bundles
$\mathbb{R}^{n}\times(\partial\Sigma)_i\times {1}$ and
$\mathbb{R}^{n}\times(\partial\Sigma)_i\times {0}$ using a clutching map $g_i\!:
(\partial\Si)_i \ri \SO (n)$ so that the bundle pair
$$(\C^{n}\times (\partial\Si)_i,\mathbb{R}^{n}\times
(\partial\Si)_i)\times_{(g_i,g_i)} I \rightarrow (\partial\Sigma)_i\times S^1$$
  is  isomorphic to $(E_{|(\partial\Si)_i\times S^1},F_{|(\partial\Si)_i\times
S^1})$.   Since the inclusion $\SO(n)\rightarrow \U(n)$ is nullhomotopic, the
map $ g_i$ can be extended to  a map
   $$\wt g_i\!: [0,2\epsilon]\times (\partial\Si)_i\ri \U (n)\qquad
\text{s.t.}\qquad \wt g_{i|[\frac{\epsilon}{2},2\epsilon]}=\id.$$

For every  $t\in S^1$,     pinch $\Sigma\times t$ along the curve
$\epsilon\times (\partial\Sigma)_i\times t$ to obtain a nodal curve $\Sigma^s$
with normalization consisting of a disjoint union of disks $D_i^2$ and a closed
Riemann surface $\hat{\Si},$ with special points $0\in D_i^2$ and $p_i\in
\hat{\Si}$. The bundle pair $(E,F)$ descends to a bundle pair over the family of
nodal curves as in Remark \ref{rem},  inducing bundles
$$
\hat{E}\rightarrow \hat\Si\times S^1\quad \text{and}\quad (E_i,F_i)\equiv
(\mathbb{C}^{n}\times D_i^2, \mathbb{R}^{n}\times (\prt
D^2)_i)\times_{(\widetilde{g}_i,g_i)}I \rightarrow (D^2_i, (\prt D^2)_i)\times
S^1,
$$
 with isomorphisms $\hat{E}_{|p_i\times t}\cong \mathbb{C}^{n} \cong
E_{i|0\times t}$  for every $t\in S^1$. \\

We are interested in the first Stiefel-Whitney class of the family of real
Cauchy-Riemann operators $D_{(E,F)}$. Taking a family of complex linear
Cauchy-Riemann operators $\hat D$ on $\hat{E}$  and gluing it to a family of
real Cauchy-Riemann operators $D_i$ on $(E_i,F_i)$, we obtain   a family of real
Cauchy-Riemann operators $D^{\varepsilon}$ on $(E,F)$. By Remark \ref{con_rmk}
and \eqref{dets},
 \begin{gather*}
 \det(D_{(E,F)})\cong\det(D^\varepsilon)\cong \det(\hat D)\otimes\bigotimes_i
\big( \det(D_i)\otimes \Lambda^{\textnormal{top}}(\hat{E}_{|(p_i,t)})\big)
 \end{gather*}
 and thus
 $$w_1(\det(D))= w_1(\det(\hat
D))+\sum_i\big(w_1(\det(D_i))+w_1(\hat{E}_{p_i\times S^1})\big).$$
The complex structure on the kernels and cokernels of the family of operators
$\hat D$ induces a canonical orientation on $\det(\hat D)$; in particular,
$w_1(\det(\hat D))$ is zero. Moreover, $\hat{E}_{p_i\times S_1}\cong
\mathbb{C}^{n}\times S^1$ also has a canonical orientation and
$w_1(\hat{E}_{p_i\times S_1})$ is   zero. Therefore, the problem reduces to the
families of operators $D_i$ on $(E_i,F_i)$ over  $(D_i^2,(\prt D^2_i))\times
S^1$.   Lemma \ref{mainord_lm}  now gives the result.
\end{proof}

\begin{proof}[{\bf\emph{Proof of Proposition \ref{mainlb_prop}}}]
For each boundary component $(\partial\Si)_i$ of $\Si$, let  $\wt U_i(\delta)$
be as in the proof of Proposition \ref{mainor_prop}. Let $m_i\in \{0,1\}$ be
equal to $0$ if $\lr{w_1(F),(\partial\Sigma)_i}=0$ and $1$ if
$\lr{w_1(F),(\partial\Sigma)_i}=1$. Then there is an isomorphism
$$(E, F)_{|\wt U_i(0)}\cong (\C\times \wt U_i(0), \Lambda_{i}\times I),$$
where the fiber of $\Lambda_i$ at a point $e^{\mathfrak{i}\theta}\times t\in
(\partial\Sigma)_i\times I$ is given by $e^\frac{\mathfrak{i}\theta
m_i}{2}\mathbb{R}\subset \C$. We can extend the trivialization $E_{|\wt
U_i(0)}\cong\C\times \wt U_i(0)$ to the neighborhood $\wt U_i(2\epsilon)$.\\

At the two endpoints of the interval $I$, glue the   bundles $\Lambda_i\times
{1}$ and $\Lambda_i\times {0}$   using a clutching map $g_i\!: (\partial\Si)_i
\ri \{\pm 1\}$ so that the bundle pair
$$(\C\times (\partial\Si)_i,  \Lambda_i)\times_{(g_i,g_i)} I \rightarrow
(\partial\Sigma)_i\times S^1$$
  is  isomorphic to $(E_{|(\partial\Si)_i\times S^1},F_{|(\partial\Si)_i\times
S^1})$.   Since the inclusion $\textnormal{O}(1)\rightarrow \U(1)$ is
nullhomotopic, the map $ g_i$ can be extended to  a map
   $$\wt g_i\!: [0,2\epsilon]\times (\partial\Si)_i\ri \U (1)\qquad
\text{s.t.}\qquad \wt g_{i|[\frac{\epsilon}{2},2\epsilon]}=\id.$$

For every  $t\in S^1$,     pinch $\Sigma\times t$ along the curve
$\epsilon\times (\partial\Sigma)_i\times t$ to obtain a nodal curve $\Sigma^s$
as in the proof of Proposition \ref{mainor_prop}.   The bundle pair $(E,F)$
descends to a bundle pair over the family of nodal curves as in Remark
\ref{rem},  inducing bundles
$$
\hat{E}\rightarrow \hat\Si\times S^1\quad \text{and}\quad (E_i,F_i)\equiv
(\mathbb{C}\times D_i^2, \Lambda_i)\times_{(\widetilde{g}_i,g_i)}I \rightarrow
(D^2_i, (\prt D^2)_i)\times S^1,
$$
 with isomorphisms $\hat{E}_{|p_i\times t}\cong \mathbb{C} \cong E_{i|0\times
t}$  for every $t\in S^1$. \\

 Taking a family of complex linear Cauchy-Riemann operators $\hat D$ on
$\hat{E}$  and gluing it to a family of real Cauchy-Riemann operators $D_i$ on
$(E_i,F_i)$, we obtain   a family of real Cauchy-Riemann operators
$D^{\varepsilon}$ on $(E,F)$. By Remark \ref{con_rmk} and \eqref{dets},
 \begin{gather*}
 \det(D_{(E,F)})\cong\det(D^\varepsilon)\cong \det(\hat D)\otimes\bigotimes_i
\big( \det(D_i)\otimes \Lambda^{\max}(\hat{E}_{|(p_i,t)})\big)
 \end{gather*}
 and thus
 $$w_1(\det(D))= w_1(\det(\hat
D))+\sum_i\big(w_1(\det(D_i))+w_1(\hat{E}_{p_i\times S^1})\big).$$
As in the proof of Proposition \ref{mainor_prop}, $w_1(\det(\hat D))$ and
$w_1(\hat E_{p_i\times S^1})$ vanish.  Thus, the problem reduces to
the families of operators $D_i$ on $(E_i,F_i)$ over  $(D_i^2,(\prt D^2_i))\times
S^1$.   Lemma \ref{mainlbd_lm}  now gives the result.
\end{proof}

\begin{proof}[{\bf \emph{Proof of Corollary \ref{rp}}}] Let $\beta_i$ be as in Corollary \ref{orient_cor}. The sum  $\sum \beta_i\in H_2(L)$ is the boundary of the class
$S\in H_3(M,L)$ obtained by following the image in $(M,L)$ of the whole surface
along the loop $\gamma$, that is $\sum\beta_i =\partial S$. 
Let
\[\dots\ri H^2(M;\mathbb{Z}_2)\overset{i^*}\longrightarrow
H^2(L;\mathbb{Z}_2)\overset{\delta}\longrightarrow H^3(M,L;\mathbb{Z}_2)\ri\dots\]
be the exact sequence for the pair $(M,L)$.  
Since  $w_2(F)+w_1^2(F)$ or $w_2(F)$  is in
the image of $i^*$, $\delta(w_2(F)+w_1^2(F))=0$ or $\delta(w_2(F))=0$. Since  $\beta_i$ is the class of a torus,
$\lr{w_1^2(F), \beta_i}=0$. Thus,
$$0=\lr{\delta(w_2(F)+w_1^2(F)),\left[S\right]}=\lr{\delta(w_2(F)), \left[S\right]} + \lr{w_1^2(F),
\sum\beta_i}=\lr{w_2(F), \sum\beta_i}.$$  
The formula  (\ref{orient_eq}) thus reduces to (\ref{pin_eq}).
\end{proof}

\section{Local systems of determinant line bundles}\label{sec:ls}

In this section we recall the basics of local systems following  \cite{Ste} and  construct a local system
$\mathcal{Z}_{(w_1,w_2)}^F$ over the $h$-fold  product of the Lagrangian $L$ and its
free loop space $\mathcal{L}(L)$.\footnote{Recall that $h$ is the
number of boundary components of $\Sigma$.} We then show its
pull-back  is canonically isomorphic to the local system twisted by the first
Stiefel-Whitney class of $\det(D_{(E,F)})$. We conclude  this section by establishing Corollaries  \ref{ls_cor} and \ref{some_cor}.

\begin{definition} A \textsf{system of local groups} $\mathcal{G}$ on a path-connected
topological space $L$ consists of
\begin{itemize}
  \item a group $G_x$ for every $x\in L$ and
  \item a group isomorphism
$\alpha_{xy}:G_x\rightarrow G_y$  for every homotopy class $\alpha_{xy}$ of
paths from $x$ to $y$
\end{itemize}
such that
 the composition $\beta_{yz}\circ\alpha_{xy}$ is the isomorphism corresponding
to the path $\alpha_{xy}\beta_{yz}$. \end{definition}

\begin{lem}[{\cite[Theorem 1]{Ste}}]\label{uniq} Suppose $p_0\!\in\! L$, $G$ is  a
group, and $\psi\!:\!\pi_1(L,p_0)\!\rightarrow\! \aut (G)$ is
a group homomorphism. Then there is a unique system
$\mathcal{G}=\{G_x\}$ of local groups on $L$ such that $G_0=G$ and the
operations of $\pi_1(L,p_0)$ on $G_0$ are those
determined by~$\psi$.
\end{lem}

Two local system $\mathcal{G}$ and $\mathcal{G}'$ on $L$ are isomorphic if for
every point $x\in L$ there is an isomorphism $h_x: G_x\cong G_x'$ such that $h_x=\alpha_{xy}^{-1} h_y
\alpha_{xy}$ for every path
$\alpha_{xy}$  between $x$ and $y$. Equivalently, two local system are isomorphic if the groups $G$
and $G'$ are isomorphic and the induced actions of $\pi_1(L,x_0)$ are the same.
There are $\aut(G)$  of such isomorphisms, and one is fixed by a choice of an
isomorphism $G_{x_0}\cong G'_{x_0}$.

\begin{definition}
Let $f\!:(L_1, p_1)\rightarrow (L_2,p_2)$ be a continuous map between
path-connected topological spaces  and let $\mathcal{G}=\{G_x\}$ be the local
system on $L_2$ induced by $\psi: \pi_1(L_2,p_2)\rightarrow \aut(G)$. The
\textsf{pull-back system} $\mathcal{G}'=f^*\mathcal{G}$ is the system of local
groups induced by $f_{\#}\circ \psi: \pi_1(L_1, p_1)\rightarrow \aut(G)$, where
$f_{\#}: \pi_1(L_1,p_1)\rightarrow \pi_1(L_2,p_2)$.
\end{definition}

\begin{definition} The \textsf{local system of orientations} for a vector bundle
$F\ri L$ is the system
induced  by the homomorphism
$\psi:\pi_1(L,p_0)\rightarrow \aut(\mathbb{Z})=\mathbb{Z}_2$ assigning to
$\alpha\in\pi_1(L,p_0)$ the value of $\lr{w_1(F),\alpha}$.
We denote this system by $\mathcal{Z}_{w_1(F)}$.
\end{definition}

\begin{remark} If $L$ is a smooth manifold, the local system of orientations for
$TL$, $\mathcal{Z}_{w_1(TL)}$,  is called the system of twisted integer
coefficients in \cite{Ste}. \end{remark}

Given a vector bundle $F\ri L$,
we now construct
a local system $\mathcal{Z}^F_{(w_1,w_2)}$ on the $h$-fold product of the
Lagrangian and its free loop space $\lL$, which traces the twisting coming from
the right-hand side of (\ref{orient_eq}).
 We begin by constructing a system over every component of
$L\times\mathcal{L}(L)$ and thus define a system over $L\times\mathcal{L}(L)$.
 We then  pull-back $h$ copies of it to the product $(L\times\mathcal{L}(L))^h$
via the projection maps. The system $\mathcal{Z}^F_{(w_1,w_2)}$ is defined as
the tensor product of the pulled-back systems.\\

Let $p_i\times\gamma_j$ be a basepoint for a component $L_i\times
\mathcal{L}(L)_j\subset L\times\lL$. Define a local system over $L_i\times
\lL_j$ using the homomorphism
\begin{gather*}
\psi:\pi_1(L_i\times\lL_j,
p_i\times\gamma_j)=\pi_1(L_i,p_i)\times\pi_1(\lL_j,\gamma_j)\rightarrow
\aut(\mathbb{Z})=\mathbb{Z}_2,\\
(\alpha,\beta)\mapsto (\lr{w_1(F), \gamma_j}+1) \lr{w_1(F),
\alpha}+\lr{w_2(F),\beta}.
\end{gather*}
 Thus, the system $\mathcal{Z}^F_{(w_1,w_2)}$ over a component of
$(L\times\lL)^h$ with a basepoint $(\vec{p},\vec{\gamma})=(p_1,\gamma_1,..\,
,p_h,\gamma_h)$ is given by the homomorphism
\begin{gather}
\psi: \pi_1((L\times\lL)^h, (\vec{p},\vec{\gamma}))\rightarrow
\aut(\mathbb{Z})=\mathbb{Z}_2, \notag\\ \label{zf_eq}
(\alpha_1,\beta_1,\dots,\alpha_h,\beta_h)\mapsto \sum_{i=1}^h (\lr{w_1(F), \gamma_i}+1)\lr{w_1(F),
\alpha_i}+\sum_{i=1}^h \lr{w_2(F),\beta_i}.
\end{gather}

We next construct a natural isomorphism between the local system
$\mathcal{Z}_{w_1(\det D_{(E,F)})}$ on $\mhp$  and a pull-back/push-down of
$\tz{(w_1,w_2)}^F$.

\begin{prop} \label{evmap_prop}Suppose there is at least one boundary marked point on each boundary
component of $\Si$. Then there is a map $\ev: \mhp\ri (L\times \lL)^h$ such that
for every bundle pair $(E,F)\ri (M,L)$ the local system
$\tz{w_1(\det(D_{(E,F)}))}$ is isomorphic to the pulled-back system
$\ev^*\tz{(w_1,w_2)}^F$.\end{prop}

\begin{proof} The map $\ev$ to the $i$-th $L$ factor is given by the evaluation at the
first marked point on the $i$-th boundary component. We now construct the map to
the $\lL$ factors. Denote by $\mathcal{D}_b$ and $\cD_{x_1}$ the groups
of  orientation-preserving diffeomorphisms   of $\Si$
that restrict to the identity on $\prt\Si$ and fix a point $x_{1,i}$ on each
component of $\prt\Si$, respectively.
The group $\cD_b$ is  a normal subgroup of $\mathcal{D}_{x_1}$, and the quotient $\mathcal{D}_{x_1}/\mathcal{D}_{b}$
is contractible. Thus,
$$(\mbp\times\mathcal{J}_\Si)/\mathcal{D}_b\ri
(\mbp\times\mathcal{J}_\Si)/\mathcal{D}_{x_1}
$$
has a contractible fiber and we can
choose a section $s$. Any two such sections are homotopic. Since the elements of
$\mathcal{D}_b$ fix the boundary of $\Si$ pointwise, there is  a map
$$e_i\!: (\mbp\times\mathcal{J}_\Si)/\mathcal{D}_b\ri \lL,\quad
[u,{\bf z},{\bf x_1},\dots,{\bf x_h}]\mapsto u_{|(\prt\Si)_i}.$$
The
map  to the $i$-th $\lL$ factor $\mhp$
is the
composition $$\ev_i\!:\mhp\cong (\mathfrak{B}_{l,{\bf k-1}}(\Si,{\bf
b})\times\mathcal{J}_\Si)/\mathcal{D}_{x_1}\overset{e_i\circ s
}{\longrightarrow}\lL.$$

We restrict ourselves to a particular connected component. Let $u_0\in \mhp$ map under
$\ev$ to the basepoint  $\vec{p}\times \vec{\gamma}\in (L\times {\lL})^h$.
It is enough to show that the action of $\pi_1(\mhp,u_0)$ on the group
$\mathbb{Z}_{u_0}$ induced by the system $\tz{w_1(\det(D_{(E,F)}))}$ is the same
as the one induced by the pulled-back system $\ev^*\tz{(w_1,w_2)}^F$.\\

By definition, the action induced by $\tz{w_1(\det(D_{(E,F)}))}$ is given by
$\lr{w_1(\det(D_{(E,F)})),\gamma}$ for $\gamma\in\pi_1(\mhp,u_0)$. By Corollary
\ref{orient_cor},
\[
  \lr{w_1(\det D_{(E,F)}),\gamma}=\sum_{i=1}^h \big(\blr{w_1(F), b_i}+1\big)
\lr{w_1(F),\alpha_i}+
\sum_{i=1}^h \blr{w_2(F),\beta_i},\]
where $\beta_i$ is the torus in $L$ traced by the $i$-th boundary $(\prt\Si)_i$
 and  $\alpha_i$ is the loop in $L$ traced by the boundary marked point
$x_{1,i}$.\\

The action of $\gamma\in\pi_1(\mhp,u_0)$ induced by the pull-back system is by
definition the action of the image $\widetilde{\gamma}\in \pi_1( (L\times
{\lL})^h,(\vec{p}_0,\vec{\gamma}_0))$  of $\gamma$ under the   map $\ev$. By
construction, it is given by
 \[ \sum_{i=1}^h (\lr{w_1(F), \gamma_i}+1)\lr{w_1(F),
\alpha_i}+\sum_{i=1}^h \lr{w_2(F),\beta_i} \]
This shows the two actions are the same.
\end{proof}

\begin{lem}\label{push_lm}
Let $\mathfrak{f}:\mhp\ri\mathcal{H}^{g,h}_{l,{\bf 0}}(M,L,{\bf b})$ be the map forgetting the boundary marked points. The system $\ev^*\mathcal{Z}^F_{(w_1,w_2)}$ pushes down to a system $$\wt\ev^*\mathcal{Z}_{(w_1,w_2)}^F=\mathfrak{f}_*\circ\ev^*\mathcal{Z}_{(w_1,w_2)}^F$$ over $\mathcal{H}^{g,h}_{l,{\bf 0}}(M,L,{\bf b})$ isomorphic to $\mathcal{Z}_{w_1(\det D_{(E,F)})}$.
\end{lem}

\begin{proof} It is enough to show that the system pushes down under the map forgetting the boundary points on a particular boundary component $(\prt\Si)_i$. Since the fiber of the forgetful map is connected,  we   need to show only that the system is trivial along the fiber. The fiber is homotopic to $S^1$. Let $\gamma$ be a loop in the fiber. Its image under the map $\ev$ is a degenerate torus, since the image of any point in the fiber is the same loop in $L$ up to reparametrization. 
 Thus, the  $w_2(F)$ term in (\ref{zf_eq}) vanishes. The image that a point on the
boundary traces along the loop is the boundary itself, and therefore the
remaining term in (\ref{zf_eq})   becomes 
$$(\lr{w_1(F), b_i}
+1)\lr{w_1(F), b_i}\equiv 0.$$Thus, the system is trivial along the fiber.\\

 Since $\pi_1(\mhp)$ surjects on $\pi_1( \mathcal{H}^{g,h}_{l,{\bf 0}}(M,L,{\bf b}) )$,  every loop $\gamma$  in $ \mathcal{H}^{g,h}_{l,{\bf 0}}(M,L,{\bf b})$ lifts to a loop $\wt\gamma$ in $\mhp$ and 
$$
\lr{w_1(\det(D_{(E,F)})),\gamma}=\lr{w_1(\det(D_{(E,F)})),\wt\gamma}.
$$
  By Proposition \ref{evmap_prop}, the induced action of $\mathcal{Z}_{w_1(\det(D_{(E,F)}))}$ on $\pi_1(\mhp)$ is the same as the one induced by $\ev^*\mathcal{Z}^F_{(w_1, w_2)}$. Thus,  the local systems $\wt\ev^*\mathcal{Z}^F_{(w_1,w_2)}$ and   $\mathcal{Z}_{w_1(
\det D_{(E,F)})}$ are isomorphic.
\end{proof}

In order to describe an isomorphism between the local systems of Proposition \ref{evmap_prop}, we   choose a trivialization of the determinant
line  $\det(D_{(E,F)})$ over $u_0$. This fixes the group $\mathbb{Z}$ at
$u_0$ and thus an isomorphism between the two systems.

\begin{prop} \label{bp_prop} A trivialization of $\det(D_{(E,F)})_{|u_0}$ is induced by
trivializations of $F^1=\det(F)$ over $u_0(x_{1,i})$ for some
$x_{1,i}\in(\partial\Sigma)_i$ and trivializations of $\wt F=F\oplus 3F^1$  over  $u_0((\partial\Sigma)_i)$ for $i=1,\dots,h$. The effect on the orientation of
$\det (D_{(E,F)})_{|u_0}$ under  the changes $s^{\wt F}_i\in \pi_1(\textnormal{SO}(n))\cong \{0,1\}$ and
$o^{F^1}_i\in \{0,1\}$  in the trivializations of $\wt F_{|u_0((\prt\Si)_i)}$ and  of $F^1_{|u_0(x_{1,i})}$ is
 the multiplication by $(-1)^\eps$, where
$$\eps=s^{\wt F}_i+(\lr{w_1(F^1), b_i}+1) o^{F^1}_i.$$
\end{prop}
\begin{proof}
By (\ref{sum}),
we have a canonical isomorphism
\begin{gather*}\det(D_{(E,F)})\otimes\det(\db{4E^1}{4F^1})\cong \det(\db{\wt
E}{\wt F}) \otimes \det(\db{E^1}{F^1}).\end{gather*}
 Thus, a choice of trivializations over $u_0$ of $\det(\db{4E^1}{4F^1})$,
$\det(\db{\wt E}{\wt F})$, and $\det(\db{E^1}{F^1})$   induces one on
$\det(D_{(E,F)})$.\\

By the proof of Proposition \ref{mainor_prop},  $\det(\db{\wt E}{\wt F})_{|u_0}$ is canonically isomorphic to the
determinant lines of  operators over  disks tensored with the determinant line
over a closed surface and the top exterior powers of  complex bundles. The last two
have canonical orientations coming from the complex structures, and thus we only
need to choose a trivialization of the determinants over the disks. The bundle
pairs over   the disks are trivial. A trivialization of $\wt F$ over each $(\prt\Si)_i$ determines a trivialization of $(\wt E,\wt F)$ over the corresponding disk $D^2_i$, uniquely up to homotopy; see the proof of Lemma \ref{mainord_lm}. The resulting trivialization of $(\wt E, \wt F)$ identifies each determinant line with the determinant line of the standard Cauchy-Riemann operator over the disk, which is canonically oriented.  This implies that a
choice of trivializations of $\wt F$  over  $u_0((\prt\Si)_i)$, with $i=1,\dots,h$,   induces a trivialization of the determinant
line  $\det(\db{\wt E}{\wt F})_{|u_0}$. Changing the homotopy type of the trivialization of $\wt F_{|u_0((\prt\Si)_i)}$ changes the induced orientation of  the determinant line over  $D^2_i$ and thus of $\det \bp_{(\wt E,\wt F)}$; see the proof of Lemma \ref{mainord_lm}. \\

Likewise, a choice of trivializations of
$4F^1$ over   $u_0((\prt\Si)_i)$, with $i=1,\dots,h$,  induces a
trivialization of the determinant line $\det(\db{4E^1}{4F^1})$. However, the bundle $4F^1$ has a canonical (up to homotopy) trivialization over each $(\prt\Si)_i$, which is induced by any trivialization of $2F^1$ over $(\prt\Si)_i$. Thus,  $\det \bp_{(4E^1, 4F^1)}$ has a canonical orientation. \\

By the proof of Proposition \ref{mainlb_prop},  $\det(\db{E^1}{F^1})_{|u_0}$ is isomorphic to the determinant lines
of operators over  disks tensored with the determinant line over a closed
surface and the top exterior powers of  complex vector bundles. The last two are
canonically oriented, and again we only need to choose a trivialization of the
determinant lines  over the disks. By Lemma \ref{lemma}, if $\lr{w_1(F^1), b_i}=0$, the index of the operator on $D^2_i$ is isomorphic to $F^1_{|u_0(x_{1,i})}$. Hence, a
choice of a trivialization of $F^1_{|u_0(x_{1,i})}$ induces a trivialization of
its determinant line; changing the homotopy type of the trivialization of $F^1_{|u_0(x_{1,i})}$ changes the induced orientation of  the determinant line over  $D^2_i$ and thus of $\det \bp_{( E^1, F^1)}$. If $\lr{w_1(F^1), b_i}=1$, the index is
isomorphic to the direct sum of the fibers of $F^1$ over the images of two points
$x_{1,i},x_{2,i}\in(\partial\Sigma)_i$. We can use the orientation of the
boundary of $\Sigma$ to transport a choice of a trivialization of
$F^1_{|u_0(x_{1,i})}$ to $F^1_{|u_0(x_{2,i})}$. In this way again, a choice of
trivializations of $F^1_{|u_0(x_{1,i})}$   determines  trivializations of the
determinant lines of the operators on the disks. However, in this case, changing the homotopy type of the trivialization of $F^1_{u_0(x_{1,i})}$ does not change the induced orientation of  the determinant line over  $D^2_i$ and thus of $\det \bp_{( E^1, F^1)}$.
\end{proof}
 
\begin{proof}[{\bf\emph{Proof of Theorem \ref{orient_thm}}}] 
By Proposition \ref{evmap_prop}, there is an isomorphism 
$$
 \mathfrak{p}\!: \mathcal{Z}_{w_1(\det D_{(E,F)})}\cong \ev^*\mathcal{Z}^F_{(w_1,w_2)}.
 $$
 By Proposition \ref{bp_prop}, a choice of a basepoint $u_0\in \mhp$, with image under the map $\ev$ equal to the basepoint of  $(L\times\lL)^h$, determines an isomorphism. We claim    that the isomorphism is independent of the choice of such $u_0$. Let us first describe the isomorphism $\mathfrak{p}$ for a given $u_0$. The trivializations of $\wt F$ and $F^1$ over $u_0((\partial\Si)_i)$ and $u_0(x_{1,i})$,  for $i=1,\dots, h$, respectively, induce an isomorphism $\mathcal{Z}_{w_1(\det D_{(E,F)})|u_0}\cong \mathbb{Z}$. By construction $\ev^*\mathcal{Z}^F_{(w_1,w_2)|u_0}=\mathbb{Z}$ and   
 $$\mathcal{Z}_{w_1(\det D_{(E,F)})|u_0}\cong \mathbb{Z}\overset{\mathfrak{p}}{\longrightarrow}\mathbb{Z}=\ev^*\mathcal{Z}^F_{(w_1,w_2)|u_0},\quad
 \mathfrak{p}(1)=1.
 $$ If $v\in \mathcal{Z}_{w_1(\det D_{(E,F)})|u}$, then  $\mathfrak{p}(v)=\psi_{\gamma^{-1}}\mathfrak{p}(\phi_\gamma v)$, where $\gamma$ is a path from $u$ to $u_0$, $\phi_\gamma$ is the isomorphism induced by the path in the system $\mathcal{Z}_{w_1(\det D_{(E,F)})}$, and $\psi_\gamma$ is the isomorphism in $\ev^*\mathcal{Z}_{(w_1,w_2)}^F$. This is independent of the path $\gamma$.\\

Let $u'_0$ be another point, with image under the map $\ev$ equal to the basepoint of $(L\times\lL)^h$,  and denote by $\mathfrak{p}'$   the induced isomorphism. We   show that 
$$
\mathfrak{p}(v)=\mathfrak{p}'(v) \quad \forall~  u\in\mhp, \,v\in \mathcal{Z}_{w_1(\det D_{(E,F)})|u}.
$$
 It is enough to confirm this    equality   for $u=u'_0$ and   $v=1\in \mathcal{Z}_{w_1(\det D_{(E,F)})|u'_0}\cong \mathbb{Z}$. Since  $\ev(u_0)=\ev(u'_0)$,  the   isomorphism 
$$
\mathbb{Z}= \mathcal{Z}_{w_1(\det D_{(E,F)})|u_0'}\cong\mathcal{Z}_{w_1(\det D_{(E,F)})|u_0}= \mathbb{Z},
$$ induced by a path $\gamma$ between $u_0'$ and $u_0$
is given by multiplication with $(-1)^\eps$, where
 $$
 \eps =\sum_{i=1}^h (\lr{w_1(F),b_i}+1)\lr{w_1(F), \ev_{x_{1,i}}(\gamma)}+\sum_{i=1}^h \lr{w_2(F),\ev_i(\gamma)},
$$ since this expression traces the change in trivializations of $\wt F$ and $F^1$. By definition, the isomorphism $\psi_\gamma$ equals $\varphi_{\ev(\gamma)}$, where $\varphi_{\ev(\gamma)}$ is the isomorphism in $\mathcal{Z}^F_{(w_1,w_2)}$ corresponding to the loop $\ev(\gamma)$.  By definition,  $\varphi_{\ev(\gamma)}$ is   also the multiplication by  $(-1)^\epsilon$. Therefore, $$\mathfrak{p}(v)=\psi_{\gamma^{-1}}\mathfrak{p}(\phi_\gamma(1))=\varphi_{\ev(\gamma)}\mathfrak{p}((-1)^\epsilon \cdot 1)=(-1)^\epsilon(-1)^\eps \mathfrak{p}(1)=1=\mathfrak{p}'(v).$$
This concludes the proof of Theorem \ref{orient_thm}.
\end{proof}

\begin{remark}\label{compatible_rem} We can choose the basepoint of $\mhp$ for different $(l,{\bf k})$ in a systematic way, thus fixing an isomorphism 
$$
\mathfrak{f}^*\mathcal{Z}_{w_1(\det(D_{(E,F)}))}\cong \mathcal{Z}_{w_1(\det(D_{(E,F)}))}.
$$
We first choose elements in $\mbp\times\cJ_\Si$ inductively. Let $u_0\in \mathfrak{B}^{g,h}_{0,{\bf 0}}(\Si,{\bf b})\times \cJ_\Si$ be a map with $u_{0|(\prt\Si)_i}=\gamma_i$. If we have chosen an element in $\mbp\times \cJ_\Si$, select an element in the space with an additional marked point $\mathfrak{B}^{g,h}_{l',{\bf k}'}(\Si,{\bf b})\times \cJ_\Si$, $l'+|{\bf k}'|=l+|{\bf k}|+1$, by adding a marked point to the given collection.   We choose the basepoint $[u_0,{\bf z},{\bf x_1},\dots,{\bf x_h}]$ of $\mhp$ to be the class of the chosen element in $\mbp\times\cJ_\Si$ and   construct the map $$\ev\!:\mhp\ri(L\times \lL)^h$$ to send $[u_0,{\bf z},{\bf x_1},\dots,{\bf x_h}]$ to the basepoint of the corresponding component.  
\end{remark}

 \begin{proof}[{\bf \emph{Proof of Corollary \ref{ls_cor}}}] Suppose   the moduli space of maps $\mmp$ and the moduli space of domains $\dmp$ are manifolds. The map  $f\!:\mmp\ri \dmp$ then induces a canonical isomorphism 
 \begin{equation*}\begin{split}
 \Lambda^\textnormal{top}T\mmp&\cong f^*\Lambda^\textnormal{top}T\dmp\otimes \Lambda^\textnormal{top}T\mmp^{\textnormal{Vert}}\\&\cong f^*\Lambda^\textnormal{top}T\dmp\otimes \Lambda^\textnormal{top}\ker D_{\bp}.
 \end{split}
 \end{equation*}
 Thus, the local system of orientations on the moduli space of maps is the restriction of the local system of orientations of the bundle $$f^*\Lambda^\textnormal{top}\dmp\otimes \det D_{\bp}\ri\mhp.$$
 
 The moduli space of domains, $\dmp$, is canonically oriented as follows.  
 Let $\mathfrak{M}^{g,h}_{l,{\bf k}'}$ be the moduli space  with one more boundary marked point. The map $\mathfrak{f}$ forgetting the additional marked point induces
 a canonical isomorphism 
  $$
  \Lambda^{\textnormal{top}}T\mathfrak{M}^{g,h}_{l,{\bf k}'}\cong \mathfrak{f}^*\Lambda^{\textnormal{top}}T\dmp\otimes  \Lambda^{\textnormal{top}}(T\mathfrak{M}^{g,h}_{l,{\bf k}'})^{\textnormal{Vert}}.
  $$
 The fiber of the map forgetting the boundary marked point is canonically isomorphic to a subset of the   boundary component the point lies on and is thus canonically oriented. 
 The transition maps of $\Lambda^{\textnormal{top}}(T\mathfrak{M}^{g,h}_{l,{\bf k}'})^{\textnormal{Vert}}$ 
 are diffeomorphisms of the domain preserving the orientation and each boundary component. Therefore, the bundle $\Lambda^{\textnormal{top}}(T\mathfrak{M}^{g,h}_{l,{\bf k}'})^{\textnormal{Vert}}$
  is orientable and canonically oriented by the orientation of a fiber. This implies that if either $\dmp$ or $\mathfrak{M}^{g,h}_{l,{\bf k}'}$ 
 has a canonical orientation, so does the other. By \cite{IS}, the moduli space $\dmp$, with ${\bf k}=(1,\dots,1)$, possesses a holomorphic structure and in particular
  is canonically oriented.
  Therefore, $\dmp$ 
  is canonically oriented   and the orientation is compatible with the maps forgetting the marked points.  Thus,
  the bundle $f^*\Lambda^\textnormal{top}\dmp\otimes \det D_{\bp}$ is canonically isomorphic to $\det D_{\bp}$. Theorem \ref{orient_thm} now implies the result.\\
  
  The case when $\mmp$ is a manifold but $\dmp$ is not can be treated as follows. Consider the moduli space $\mathfrak{M}^{g,h}_{l',{\bf k}'}(M,L,{\bf b})$ with enough marked points so that $\mathfrak{M}^{g,h}_{l',{\bf k}'}$ is a manifold. Then, 
  $$\Lambda^{\textnormal{top}}T\mathfrak{M}^{g,h}_{l',{\bf k}'}(M,L,{\bf b})\cong \det D_{\bp}.$$ Moreover, the map forgetting the additional marked points
  $$ \mathfrak{f}\!:\mathfrak{M}^{g,h}_{l',{\bf k}'}(M,L,{\bf b})\ri \mmp$$
  induces a canonical isomorphism 
 $$
  \Lambda^{\textnormal{top}}T\mathfrak{M}^{g,h}_{l',{\bf k}'}(M,L,{\bf b})\cong \mathfrak{f}^*\Lambda^{\textnormal{top}}T\mmp\otimes  \Lambda^{\textnormal{top}}(T\mathfrak{M}^{g,h}_{l',{\bf k}'}(M,L,{\bf b}))^{\textnormal{Vert}}.
  $$
As above, the bundle $\Lambda^{\textnormal{top}}(T\mathfrak{M}^{g,h}_{l',{\bf k}'}(M,L,{\bf b}))^{\textnormal{Vert}}$ is canonically oriented and thus,
$$
\mathfrak{f}^*\Lambda^{\textnormal{top}}T\mmp\cong \Lambda^{\textnormal{top}}T\mathfrak{M}^{g,h}_{l',{\bf k}'}(M,L,{\bf b})\cong \det D_{\bp}\cong\mathfrak{f}^*\det D_{\bp}.
$$ 
Since $\mathfrak{f}$ is surjective on $\pi_1$, $\Lambda^{\textnormal{top}}T\mathfrak{M}^{g,h}_{l,{\bf k}}(M,L,{\bf b})\cong\det D_{\bp}$; see the  proof of Lemma~\ref{push_lm}. The result follows from Theorem \ref{orient_thm}.
 \end{proof}
  
 \begin{proof}[{\bf \emph{Proof of Corollary \ref{some_cor}}}]
 The line bundles 
 $$
\Lambda^{\textnormal{top}}_\R \mathcal{V}^\R_{n;{\bf a}}\cong \det \bp_{(\mathcal{L}_{n,{\bf a}},\mathcal{L}_{n,{\bf a}}^\R)} 
,~ \Lambda^{\textnormal{top}}_\R T\mathfrak{M}_{l,{\bf k}}^{g,h}(\C\mathbb{P}^n,\R\mathbb{P}^n,{\bf b})\ri \mathfrak{M}_{l,{\bf k}}^{g,h}(\C\mathbb{P}^n,\R\mathbb{P}^n,{\bf b})
$$
are isomorphic if  their first Stiefel-Whitney classes evaluate to  the same number over every loop $\gamma$.
 By Corollary \ref{orient_cor}, the first Stiefel-Whitney class of $\mathcal{V}_{n,{\bf a}}$ evaluated on a loop $\gamma$ is given by 
 \begin{equation}\label{wln_eq}
 \sum_{i=1}^h (\lr{w_1(\mathcal{L}_{n,{\bf a}}^\R), b_i}+1)\lr{w_1(\mathcal{L}_{n,{\bf a}}^\R),
\ev_{x_{1,i}}(\gamma)}+\sum_{i=1}^h \lr{w_2(\mathcal{L}_{n,{\bf a}}^\R),\ev_i(\gamma)} .
 \end{equation}
 
 Let $\mathcal{O}^\R(a)$ denote the tensor product of $a$ copies of the tautological line bundle over $\R\mathbb{P}^n$ and
  $\eta=w_1(\mathcal{O}^\R(1))$ be the generator of $H^1(\R\mathbb{P}^n,\mathbb{Z}_2)$.
 Since $\mathcal{L}_{n,{\bf a}}^\R=\bigoplus_i\mathcal{O}^\R(a_i)$,
 $$w_1(\mathcal{L}_{n,{\bf a}}^\R)=\sum_{i=1}^m a_i\eta, \quad w_2(\mathcal{L}_{n,{\bf a}}^\R)=\sum_{i,j}^m a_ia_j \eta^2.
 $$ 
 Since $w_2(\mathcal{L}_{n,{\bf a}}^\R)$ is a square of a class, it evaluates to zero on each torus $\ev_i(\gamma)$.  
 By Corollary \ref{orient_cor}, the first Stiefel-Whitney class of $\mathfrak{M}^{g,h}_{l,{\bf k}}(\C\mathbb{P}^n,\R\mathbb{P}^n,{\bf b})$ evaluated on the loop $\gamma$ is given by (\ref{wln_eq})  with $\mathcal{L}_{n,{\bf a}}^\R$ replaced by $T\R\mathbb{P}^n$. Since $w_2(\R\mathbb{P}^n)$ is a square of a class, the second sum vanishes in this case as well. The condition $n-\sum a_i$ is odd implies that
 $$
 w_1(\mathcal{L}_{n,{\bf a}}^\R)=(n+1)\eta=w_1(\R\mathbb{P}^n).
 $$
 Thus, the two line bundles are isomorphic.\\
  
   By Proposition \ref{bp_prop}, a choice of trivializations of $ \mathcal{L}_{n,{\bf a}}^\R\oplus3\det(\mathcal{L}_{n,{\bf a}}^\R)$ and $ \det(\mathcal{L}_{n,{\bf a}}^\R)$ over $\ev_i(u_0)$ and $u_0(x_{1,i})$, respectively, with $i=1,\dots,h$, determines a trivialization   $\det \bp_{(\mathcal{L}_{n,{\bf a}},\mathcal{L}_{n,{\bf a}}^\R)|u_0}\cong \mathbb{R}$. Similarly, a trivialization 
 $$\Lambda^{\textnormal{top}}_\R T\mathfrak{M}_{l,{\bf k}}^{g,h}(\C\mathbb{P}^n,\R\mathbb{P}^n,{\bf b})_{|u_0}\cong\mathbb{R}$$
 is determined by a choice of trivializations of $T\R\mathbb{P}^n\oplus 3\det(T\R\mathbb{P}^n)$ and $\det(T\R\mathbb{P}^n)$  over   $\ev_i(u_0)$ and $u_0(x_{1,i})$, respectively, with $i=1,\dots,h$. \\
 
  If $\mathcal{O}^\R$ denotes the trivial line bundle, there are canonical isomorphisms
 $$  
  \mathcal{O}^\R\oplus T\R\mathbb{P}^n\cong (n+1)\mathcal{O}^\R(1),\qquad \det(T\R\mathbb{P}^n)\cong\mathcal{O}^\R(n+1).
 $$
 Thus, a choice of a trivialization of $(n+1)\mathcal{O}^\R(1)\oplus 3\mathcal{O}^\R(n+1)$ over $\ev_i(u_0)$ determines one on $T\R\mathbb{P}^n\oplus 3\det(T\R\mathbb{P}^n)$. 
A choice of a trivialization of $\mathcal{O}^\R(1)$ over $u_0(x_{1,i})$   and the canonical  trivialization  
$$
2\mathcal{O}^\R(1)\cong \mathcal{O}^\R\oplus T\R\mathbb{P}^1\cong 2\mathcal{O}^\R
$$
 over $\ev_i(u_0)$  determine an isomorphism
\begin{equation}\label{oiso_eq}
 \!(n\!+\!1)\mathcal{O}^\R(1)\!\oplus \!3\mathcal{O}^\R(n\!+\!1)\!\cong\!\begin{cases} \frac{n}{2}(2\mathcal{O}^\R)\oplus\mathcal{O}^\R(1)\!\oplus \!3\mathcal{O}^\R(1)\!\cong\! (n\!+\!4)\mathcal{O}^\R,\!&\! \textnormal{if}~2|n,\\
  \frac{n+1}{2}(2\mathcal{O}^\R)\oplus3\mathcal{O}^\R,\!&\!\textnormal{if}~2\!\!\not| n,
\end{cases}\end{equation}
over $\ev_i(u_0)$.  A non-trivial change in the trivialization of $\mathcal{O}^\R(1)$ over $u_0(x_{1,i})$ does not affect the trivialization of $3\mathcal{O}^\R(n\!+\!1)$ in (\ref{oiso_eq}) when $n$ is odd. Thus, the trivialization in (\ref{oiso_eq}) is canonical.
 A choice of a trivialization of $\mathcal{O}^\R(1)$ over $u_0(x_{1,i})$ determines one on $\mathcal{O}^\R(n\!+\!1)\cong \det(T\R\mathbb{P}^n)$. A non-trivial change in the trivialization of $\mathcal{O}^\R(1)$ over $u_0(x_{1,i})$ results in a non-trivial change of trivialization of $\mathcal{O}^\R(n\!+\!1)$ over $u_0(x_{1,i})$ if and only if $n$ is even.
\\

There are canonical isomorphisms
$$
\mathcal{L}_{n,{\bf a}}^\R\cong (m-\wt m)\mathcal{O}^\R\oplus\wt m\mathcal{O}^\R(1), \qquad \det(\mathcal{L}_{n,{\bf a}}^\R)\cong \mathcal{O}^\R(\wt m),
$$
where $\wt m$ is the number of odd $a_i$.
By (\ref{oiso_eq}),   $ \mathcal{L}_{n,{\bf a}}^\R\oplus3\det(\mathcal{L}_{n,{\bf a}}^\R)$ has a canonical trivialization over $\ev_i(u_0)$.  A trivialization of  $\mathcal{O}^\R(1)$ over $u_0(x_{1,i})$ induces one on $\mathcal{O}^\R(\wt m)$, and a non-trivial change in the former results in a non-trivial change in the latter if and only if $\wt m$ is odd. Thus,  a choice of trivializations of
 $\mathcal{O}^\R(1)$ over $u_0(x_{1,i})$ determine isomorphisms
$$\Lambda^{\textnormal{top}}_\R T\mathfrak{M}_{l,{\bf k}}^{g,h}(\C\mathbb{P}^n,\R\mathbb{P}^n,{\bf b})_{|u_0}\cong\mathbb{R},\qquad  \Lambda^{\textnormal{top}}_\R \mathcal{V}^\R_{n;{\bf a}|u_0}\cong\R.$$
 If $\wt m\equiv n+1$ mod $2$, a   change in the  trivialization $\mathcal{O}^\R(1)_{|u_0(x_{1,i})}$  affects both isomorphisms in the same way. Thus, in this case, the composite isomorphism
$$\Lambda^{\textnormal{top}}_\R T\mathfrak{M}_{l,{\bf k}}^{g,h}(\C\mathbb{P}^n,\R\mathbb{P}^n,{\bf b})_{|u_0}\cong\mathbb{R}\cong  \Lambda^{\textnormal{top}}_\R \mathcal{V}^\R_{n;{\bf a}|u_0}$$
is canonical.
  \end{proof}


\begin{thebibliography}{99}
\bibitem[BCPP]{Go} G. Bini, C. de Cocini, M. Polito, and C. Procesi, \emph{On the work
of Givental relative to mirror symmetry}, Appunti dei Corsi Tenuti da Docenti
della Scuola, Scuola Normale Superiore, Pisa, 1998
 
 \bibitem[deS]{deS} V. de Silva, \emph{Products in the symplectic Floer homology of Lagrangian intersections}, Ph. D. thesis, University of Oxford, 1998
 
 \bibitem[EES]{EES} T. Ekholm, J. Etnyre, M. Sullivan, \emph{Orientations in
Legendrian contact homology and exact Lagrangian immersions}, Internat. J. Math.
16 (2005), no. 5, 453--532

\bibitem[FM]{FM}  B. Farb,  D. Margalit,  \emph{A Primer on Mapping Class Groups}, Princeton Mathematical Series 49, Princeton University Press, Princeton, NJ, 2012

 \bibitem[FOOO]{FOOO} K. Fukaya, Y.-G. Oh, H. Ohta, K. Ono,
\emph{Lagrangian Floer Theory: Anomaly and Obstruction}, AMS, Providence, RI,
2009
 
 \bibitem [Gro]{Gr} M. Gromov, \emph{Pseudoholomorphic curves in symplectic
manifolds},  Invent.
Math. 82 (1985), no. 2, 307--347
 
\bibitem[IS]{IS} S. Ivashkovich, V. Shevchishin, \emph{Holomorphic structure on the space of Riemann surfaces with marked boundary}, Tr. Mat. Inst. Steklova 235 (2001), Anal. i Geom. Vopr. Kompleks. Analiza, 98-109; translation in Proc. Steklov Inst. Math. 2001, no. 4 (235), 91-102

\bibitem[LZ]{LZ}J. Li, A. Zinger, \emph{On the genus-one Gromov-Witten invariants of
complete intersections}, JDG 82 (2009), no. 3, 641-690

 
\bibitem[Loo]{Loo} E. Looijenga, \emph{Smooth Deligne-Mumford compactifications by means of Prym level structures}, J. Algebraic Geom. 3 (1994), 283-293 

\bibitem[Mas]{Mas} G. Massuyea, \emph{A short introduction to mapping class groups}, available at http://www-irma.u-strasbg.fr/~massuyea/talks/MCG.pdf


\bibitem[MS]{MS} D. McDuff, D. Salamon, \emph{J-holomorphic curves and
symplectic topology}, volume 52 of American Mathematical Society Colloquium
Publications, American Mathematical Society, Providence, RI,
2004.

\bibitem[PSW]{psw} R. Pandharipande, J.  Solomon, J. Walcher,
\emph{Disk enumeration on the quintic 3-fold}, Journal of the American
Mathematical Society 21 (2008), no. 4, 1169-1209

\bibitem[PZ]{PZ} A. Popa, A. Zinger,  \emph{Mirror symmetry for closed, open,
and unoriented Gromov-Witten invariants}, preprint 2010, arXiv.org:math/1010.1946

 \bibitem[Sol]{Sol} J. Solomon,  \emph{Intersection theory on the moduli space
of holomorphic curves with Lagrangian boundary conditions}, preprint 2006,
arXiv.org:math/0606429

\bibitem[Ste]{Ste} N. E. Steenrod, \emph{Homology with local coefficients}, The
Annals of Mathematics, Second Series, Vol. 44, No. 4 (Oct., 1943), pp. 610-627

\bibitem[Wal]{Wal} J. Walcher, \emph{Evidence for tadpole cancellation in the
topological string}, Comm. Number Theory Phys. 3 (2009), no. 1, 111-172

\bibitem[WW]{WW} K. Wehrheim, C. Woodward, \emph{Orientations for
pseudoholomorphic quilts}, preprint 2007, available at
http://math.mit.edu/~katrin/

 \bibitem[Wel08]{Wel} J.-Y. Welschinger,  \emph{Enumerative invariants of
strongly semipositive real symplectic six-manifolds}, preprint 2008,
arXiv.org:math/0509121
\end{thebibliography}
\end{document}